\documentclass[10pt]{amsart}

\usepackage{amssymb,epsfig}

\DeclareMathOperator{\diam}{diam}
\DeclareMathOperator{\supp}{supp}
\DeclareMathOperator{\Lip}{Lip}
\newtheorem{theorem}{Theorem}[section]
\newcommand{\bthm}{\begin{theorem}}
\newcommand{\ethm}{\end{theorem}}
\newtheorem{prop}[theorem]{Proposition}
\newcommand{\bprop}{\begin{prop}}
\newcommand{\eprop}{\end{prop}}
\newtheorem{lemma}[theorem]{Lemma}
\newtheorem{oss}[theorem]{\sc Remark}
\newcommand{\boss}{ \begin{oss}  \rm -\ }
\newcommand{\eoss}{ \end{oss} }
\newtheorem{definition}[theorem]{Definition}
\newcommand{\bd}{\begin{definition}}
\newcommand{\ed}{\end{definition}}

\newcommand{\be}{\begin{equation}}
\newcommand{\ee}{\end{equation}}
\newcommand{\ba}{\begin{array}}
\newcommand{\ea}{\end{array}}

\newcommand{\dint}{\displaystyle \int}

\newcommand{\R}{{\bf R}}
\newcommand{\N}{{\bf N}}
\newcommand{\e}{\epsilon}
\newcommand{\om}{\omega}
\newcommand{\Om}{\Omega}
\newcommand{\si}{\sigma}
\newcommand{\la}{\lambda}
\newcommand{\esp}{\kappa}
\newcommand{\osc}{\mathop{\rm osc}\limits}
\newcommand{\esssup}{\mathop{\rm ess\hspace{0.1 cm}sup}\limits}
\newcommand{\essinf}{\mathop{\rm ess\hspace{0.1 cm}inf}\limits}

\newcommand{\mis}{\mu \otimes \mathcal{L}^1}
\def\mediad#1{{\int\!\!\!\!\int\!\!\!\!\!\!-\!\!\!\!\!\!\!\!-}_{{#1}}}

\title{Harnack's inequality for parabolic De Giorgi classes in metric spaces}
\author[Kinnunen, Marola, Miranda jr. and Paronetto]{ J. Kinnunen
  \address[JK]{Department of Mathematics, P.O Box 11100, FI-00076
    Aalto University, Finland, e--mail: juha.k.kinnunen@aalto.fi},
  N. Marola\address[NM]{Department of Mathematics and Statistics,
    University of Helsinki, P.O. BOX 68, FI-00014 University of
    Helsinki, e--mail: niko.marola@helsinki.fi}, M. Miranda
  jr.\address[MM]{Dipartimento di Matematica, University of Ferrara,
    via Machiavelli 35, 44121 Ferrara, Italy, e--mail:
    michele.miranda@unife.it.  }, and
  F. Paronetto\address[FP]{Dipartimento di Matematica, Universit\`a degli 
    Studi di Padova, via Trieste 63,
    35121 Padova, Italy, e--mail: fabio.paronetto@unipd.it.}}

\date{}

\begin{document}

\begin{abstract}
  In this paper we study  problems related to parabolic partial differential
  equations in metric measure spaces equipped with a doubling measure and
  supporting a Poincar\'e inequality. We give a definition of
  parabolic De Giorgi classes and compare this notion with that of
  parabolic quasiminimizers. The main result, after proving the local
  boundedness, is a scale and location invariant Harnack inequality
  for functions belonging to parabolic De Giorgi classes.
  In particular, the results hold true for parabolic quasiminimizers.
\end{abstract}

\maketitle

{\Small
{\it MSC: 30L99, 31E05, 35K05, 35K99, 49N60}
{\it Keywords:} De Giorgi class; doubling measure; Harnack inequality;
H\"ol\-der continuity; metric space; minimizer; Newtonian space;
parabolic; Poincar\'e inequality; quasiminima, quasiminimizer.
}

\section{Introduction}

The purpose of this paper is to study problems related to the heat equation
\[
\frac{\partial u}{\partial t}-\Delta u =0
\]
in metric measure spaces equipped with a doubling measure and supporting a
Poincar\'e inequality. We consider a notion of parabolic De Giorgi classes
and parabolic quasiminimizers with quadratic structure conditions and study local regularity
properties of functions belonging to these classes. More precisely, we
show that functions in parabolic De Giorgi classes,
satisfy a scale and location invariant Harnack
inequality, see Theorem~\ref{Harnack}. Some consequences of the
parabolic Harnack inequality are the local H\"older
continuity and the strong maximum principle for the parabolic De
Giorgi classes. Our assumptions on the metric space are rather
standard to allow a reasonable first-order calculus;
the reader should consult, e.g., Bj\"orn and Bj\"orn~\cite{BjoBjo} and 
Heinonen~\cite{Hei01Lec}, and the references therein.

Harnack type inequalities play an important role in the regularity
theory of solutions to both elliptic and parabolic partial
differential equations as it implies local H\"older continuity for the
solutions. A parabolic Harnack inequality is logically stronger than
an elliptic one since the reproduction at each time of the same
harmonic function is a solution of the heat equation. There is,
however, a well-known fundamental difference between elliptic and
parabolic Harnack estimates. Roughly speaking, in the elliptic case
the information of a positive solution on a ball is controlled by the
infimum on the same ball. In the parabolic case a delay in time is
needed: the information of a positive solution at a point and at
instant $t_0$ is controlled by a ball centered at the same point but
later time $t_0+t_1$, where $t_1$ depends on the parabolic equation.

Elliptic quasiminimizers were introduced by
Giaquinta--Giusti~\cite{GiaGiu82OnT} and \cite{GiaGiu84Qua} as a tool
for a unified treatment of variational integrals, elliptic equations
and systems, and quasiregular mappings on $\R^n$. Let $\Omega \subset
\R^n$ be a nonempty open set.  A function $u \in
W^{1,p}_{\rm{loc}}(\Omega)$ is a $Q$-quasiminimizer, $Q \ge 1$,
related to the power $p$ in $\Om$ if
\[
\int_{\supp(\phi)} |\nabla u|^p\, dx
   \leq Q \int_{\supp(\phi)} |\nabla (u-\phi)|^p\, dx
\]
for all $\phi \in W^{1,p}_0(\Omega)$. Giaquinta and Giusti realized
that De Giorgi's method \cite{degiorgi3} could be extended to
quasiminimizers, obtaining, in particular, local H\"older continuity.
DiBenedetto and Trudinger~\cite{DibTru84Har} proved the Harnack
inequality for quasiminimizers.  These results were extended to metric
spaces by Kinnunen and Shanmugalingam~\cite{kin-sha}.  Elliptic
quasiminimizers enable the study of elliptic problems, such as the
$p$-Laplace equation and $p$-harmonic functions, in metric
spaces under the doubling property of the measure and a Poincar\'e inequality. 
Compared with the theory of $p$-harmonic functions we have no
partial differential equation, only the variational approach is available.
There is also no comparison principle nor uniqueness for the Dirichlet
problem for quasiminimizers. See, e.g., 
J. Bj\"orn~\cite{Bjo02Bou}, Kinnunen--Martio~\cite{KinMar03Pot},
Martio--Sbordone~\cite{MarSbo07Qua} and the references in these papers
for more on elliptic quasiminimizers.

Following Giaquinta--Giusti, Wieser~\cite{wieser} generalized the
notion of quasiminimizers to the parabolic setting in Euclidean
spaces. A function $u:\Omega \times (0,T) \to \R$, $u \in L^2_{\rm
  loc}(0,T; W^{1,2}_{\rm loc}(\Omega))$, is a parabolic
$Q$-quasiminimizer, $Q\geq 1$, for the heat equation (thus related to
the power $2$) if
\[
- \iint_{\supp(\phi)} u \frac{\partial \phi}{\partial t} \,dx \,dt +
\iint_{\supp(\phi)}\frac{|\nabla u|^2}{2}\, dx\,dt \leq Q
\iint_{\supp(\phi)}\frac{|\nabla (u-\phi)|^2}{2}\, dx\,dt
\]
for every smooth compactly supported function $\phi$ in $\Omega \times
(0,T)$.  Parabolic quasiminimizers have also been studied by
Zhou~\cite{Zho93OnT, Zho94Par}, Gianazza--Vespri~\cite{GiaVes06Par},
Marchi~\cite{Mar94Bou}, and Wang~\cite{Wan89Har}. 
The literature for parabolic quasiminimizers is very small compared to the 
elliptic case. One of the goals of this work is to introduce parabolic quasiminimizers
in metric metric spaces. This opens up a possibility to develop a systematic
theory for parabolic problems in this generality.

The present paper is using the ideas of DiBenedetto~\cite{Dib89Har} and is
based on the lecture notes~\cite{ForParVes08Dis} of the course held by
V.~Vespri in Lecce. We would like to point out that the definition for 
the parabolic De Giorgi classes given by Gianazza and Vespri~\cite{GiaVes06Par} is sligthly
different from ours, and it seems that our class is larger. 
Naturally, our abstract setting causes new
difficulties. For example, Lemma~\ref{lemmaMisVar} plays a crucial
role in the proof of Harnack's inequality. In Euclidean spaces this
abstract lemma dates back to
DiBenedetto--Gianazza-Vespri~\cite{DibGiaVes06Loc}, but as the proof
uses the linear sructure of the ambient space a new proof in the
metric setting was needed.

Motivation for this work is to consider the Saloff-Coste--Grigor'yan theorem in metric measure spaces.
Grigor'yan~\cite{Gri91The} and Saloff-Coste~\cite{Sal92ANo} observed
independently that the doubling property for the measure and the
Poincar\'e inequality are sufficient and necessary conditions for a
scale invariant parabolic Harnack inequality for the heat equation on
Riemannian manifolds. Sturm~\cite{Stu96Ana} generalized this result to
the setting of local Dirichlet spaces; such approach works also in fractal geometries, but always
when a Dirichlet form is defined. For references, see for instance 
Barlow--Bass--Kumagai~\cite{BarBasKum06Sta} and also the forthcoming 
paper by Barlow--Grigor'yan--Kumagai~\cite{BarGriKum11OnT}.   

In this paper we introduce a version of parabolic De
Giorgi classes that include parabolic quasiminimizers and show the sufficiency of the 
Saloff-Coste--Grigor'yan theorem in metric measure spaces without using Dirichlet spaces nor the Cheeger
differentiable structure \cite{cheeger}.
Very recently a similar question has been studied for degenerate parabolic quasilinear 
partial differential equations in the subelliptic case by Caponga--Citti--Rea~\cite{CaCiRe}.
Their motivating example is a class of subelliptic operators associated to a family of H\"ormander vector fields
and their Carnot--Carath\'eodory distance.
We show that the doubling property and the Poincar\'e inequality implies
a scale and location invariant parabolic Harnack inequality for functions belonging
to De Giorgi classes, and thus for parabolic quasiminimizers, in general metric measure spaces.
It would be very interesting to know whether also necessity holds in this setting for De Giorgi
classes or for parabolic quasiminimizers; using the results contained in 
Sturm~\cite{Stu98Dif,Stu98How}, it is possible to contruct a regular Dirichlet form, 
and then a diffusion process, on every locally compact metric spaces, and this, combined
with Sturm~\cite{Stu96Ana} can be used to obtain the reverse implication. This is, however, 
a very abstract result and it is based on $\Gamma$-convergence of non-local Dirichlet
forms, with no information on the limiting Dirichlet form. Such geometric
characterization via the doubling property of the measure and a
Poincar\'e inequality is not available for an elliptic Harnack
inequality, see Delmotte~\cite{Del02Gra}.

The paper is organized as follows. In Section~\ref{secPrel} we recall
the definition of Newton--Sobolev spaces and prove some preliminary
technical results; these results are general results on Sobolev 
functions and are of independent interest. 
In Section~\ref{secDgQm} we introduce the parabolic
De Giorgi classes of order 2 and define parabolic
quasiminimizers.  
In Section~\ref{secLocBdd} we prove the`
local boundedness of elements in the De Giorgi classes, and finally,
in Section \ref{secHarnack} we prove a Harnack-type inequality.

\medskip

{\bf Acknowledgements} Miranda and Paronetto visited the Aalto
University School of Science and Technology in February 2010, and
Marola visited the Universit\`a di Ferrara and Universit\`a degli
studi di Padova in September 2010, and Kinnunen visited the Universit\`a 
of Padova in January 2011. It is a pleasure to thank these universities for the hospitality and
the Academy of Finland for the financial support.
This work and the visits of Kinnunen and Marola were also partially supported 
by the 2010 GNAMPA project ``Problemi geometrici, variazionali ed evolutivi 
in strutture metriche''.

\section{Preliminaries}
\label{secPrel}

In this section we briefly recall the basic definitions and collect
some results needed in the sequel. For a more detailed treatment we
refer, for instance, to the forthcoming monograph by
A. and J. Bj\"orn~\cite{BjoBjo} and the references therein.

Standing assumptions in this paper are as follows. By the triplet
$(X,d,\mu)$ we will denote a complete metric space $X$, where $d$ is
the metric and $\mu$ a Borel measure on $X$.  The measure $\mu$ is
supposed to be \emph{doubling}, i.e., there exists a constant $c
\geq 1$ such that
\begin{equation}
\label{doubling_provv}
0<\mu(B_{2r}(x)) \leq c \mu (B_r(x))<\infty
\end{equation}
for every $r > 0$ and $x \in X$. Here $B_r(x)=B(x,r)=\{y\in X:\
d(y,x)<r\}$ is the open ball centered at $x$ with radius $r>0$.
We want to mention in passing that to require the measure of every ball in $X$ to be positive and 
finite is anything but restrictive; it does not rule out any interesting measures. 
The \emph{doubling constant} of $\mu$ is defined to be
$c_d := \inf \{c \in (1,\infty): \ \eqref{doubling_provv} \textrm{
  holds true} \}$. The doubling condition implies that for any $x\in X$, we have
\begin{equation}
\label{stime_doub_e_revdoub}
\frac{\mu(B_R(x))}{\mu(B_r(x))} \leq c_{d}\left(\frac{R}{r}\right)^{N}
= 2^{N}  \left(\frac{R}{r}\right)^{N}  ,
\end{equation}
for all $0 < r\leq R$ with $N := \log_2 c_d$. The exponent $N$ serves as a counterpart of 
dimension related to the measure. 
Moreover, the product measure in the space
$X\times (0,T)$, $T>0$, is denoted by $\mis$, where $\mathcal{L}^1$ is
the one dimensional Lebesgue measure.

We follow Heinonen and Koskela~\cite{HeiKos98Qua} in introducing
upper gradients as follows. A Borel function $g: X\to [0, \infty]$
is said to be an \emph{upper gradient} for an extended real-valued function $u$ on $X$
if for all rectifiable paths $\gamma: [0,l_{\gamma}] \to X$, we have
\begin{equation}\label{defUG}
| u(\gamma(0)) - u(\gamma(l_\gamma))| \leq \int_{\gamma} g \, ds.
\end{equation}
If \eqref{defUG} holds for $p$--almost every curve, we say that $g$ is a
$p$--weak upper gradient of $u$; here by $p$--almost
every curve we mean that \eqref{defUG} fails only for a  curve family 
$\Gamma$ with zero $p$--modulus. Recall, that the $p$--modulus
of a curve family $\Gamma$ is defined as
\[
{\rm Mod}_p\Gamma = \inf \left\{ \int_X\varrho^p\,d\mu : \varrho\geq 0\text{ is a Borel function},
\int_\gamma \varrho\geq 1
\text{ for all }\gamma \in \Gamma\right\}.
\]
From the definition, it follows
immediately that if $g$ is a $p$-weak upper gradient for $u$, then $g$
is a $p$-weak upper gradient also for $u-k$, and $|k|g$ for $ku$, for
any $k\in \R$.

The $p$-weak upper gradients were introduced in
Koskela--MacManus~\cite{KosMac98Qua}. They also showed that if $g \in
L^p(X)$ is a $p$--weak upper gradient of $u$, then, for any $\varepsilon>0$, one can 
find an upper gradients $g_\varepsilon$ of $u$ such that
$g_\varepsilon> g$ and $\|g_\varepsilon-g\|_{L^p(X)}<\varepsilon$.  
Hence for most practical purposes it is
enough to consider upper gradients instead of $p$--weak upper
gradients.  If $u$ has an upper gradient in $L^p(X)$, then it has a
unique \emph{minimal $p$-weak upper gradient} $g_u \in L^p(X)$ in the sense
that for every $p$-weak upper gradient $g \in L^p(X)$ of $u$, $g_u
\leq g$ a.e., see Corollary~3.7 in
Shan\-mu\-ga\-lin\-gam~\cite{sha-harm} and Haj\l asz~\cite{Haj03Sob}
for the case $p=1$.

Let $\Omega$ be an open subset of $X$ and $1\leq p<\infty$.
Following the definition of Shanmugalingam~\cite{sha}, we define for
$u\in L^p(\Omega)$,
\[
\|u\|^p_{N^{1,p}(\Omega)}:=\|u\|^p_{L^p(\Omega)}+ \inf\|g\|^p_{L^{p}(\Omega)},
\]
where the infimum is taken over all upper gradients of $u$.  The
\emph{Newtonian space} $N^{1,p}(\Omega)$ is the quotient space
\[
N^{1,p}(\Omega)=\left\{
u\in L^{p}(\Omega): \|u\|_{N^{1,p}(\Omega)}<\infty
\right\} /{\sim},
\]
where $u\sim v$ if and only if $\|u-v\|_{N^{1,p}(\Omega)}=0$. 
If $u,v\in N^{1,p}(X)$ and $v=u$ $\mu$-almost everywhere, then $u\sim v$.
Moreover, if $u\in N^{1,p}(X)$, then $u\sim v$ if and only if $u=v$
outside a set of zero $p$-capacity.
The space $N^{1,p}(\Omega)$ is a Banach space
(see Shanmugalingam~\cite[Theorem 3.7]{sha} and it is easily
verified that it has a lattice structure. A function $u$ belongs to the \emph{local
Newtonian space} $N^{1,p}_{\rm{loc}}(\Omega)$ if $u\in N^{1,p}(V)$
for all bounded open sets $V$ with $\overline{V}\subset\Omega$, the
latter space being defined by considering $V$ as a metric space with
the metric $d$ and the measure $\mu$ restricted to it.

Newtonian spaces share many properties of the classical Sobolev
spaces.  For example, if $u,v \in N^{1,p}_{\rm{loc}}(\Omega)$, then
$g_u=g_v$ a.e.\ in $\{x \in \Omega : u(x)=v(x)\}$, in particular
$g_{\min\{u,c\}}=g_u \chi_{\{u \neq c\}}$ for $c \in \R$.

We shall also need a Newtonian space with zero boundary values; for the detailed
definition and main properties we refer to Shanmugalingam~\cite[Definition 4.1]{sha-harm}. For a
measurable set $E\subset X$, let
\[
N_0^{1,p}(E) = \{f|_E: f\in N^{1,p}(X) \textrm{ and } f= 0\; p\textrm{--a.e. on
} X\setminus E\}.
\]
The notion of $p$--a.e. is based on the notion of sets of null $p$--capacity; the
$p$--capacity of a set $E$ can be defined as
\[
{\rm Cap}_p E=\inf\left\{ \| u\|_{N^{1,p}(X)}^p: u\in N^{1,p}(X), u\geq 1 \mbox{ on }E\right\}. 
\]
This space equipped with the norm inherited from $N^{1,p}(X)$ is a
Banach space.

We shall assume that $X$ supports a \emph{weak $(1,2)$-Poincar\'e inequality},
that is there exist constants $C_2>0$ and $\Lambda \geq 1$ such that for
all balls $B_\rho \subset X$, all integrable functions $u$ on $X$ and
all upper gradients $g$ of $u$,
\begin{equation}\label{PoiIn}
{\int\!\!\!\!\!\!-}_{B_{\rho}} |u-u_{B_{\rho}}|\, d\mu
\leq C_2 \rho \left(
{\int\!\!\!\!\!\!-}_{B_{\Lambda\rho}}
g^2 \, d\mu \right)^{1/2},
\end{equation}
where
\[
u_B:= {\int\!\!\!\!\!\!-}_B u\, d\mu := \frac1{\mu(B)}\int_B u\, d\mu.
\]
It is noteworthy that by a result of Keith and
Zhong~\cite{KeiZho08Poi} if a complete metric space is equipped with a
doubling measure and supports a weak $(1,2)$-Poincar\'e inequality,
then there exists $\varepsilon > 0$ such that the space admits a weak
$(1,p)$-Poincar\'e inequality for each $p>2-\varepsilon$. We shall
use this fact in the proof of Lemma~\ref{lemma2} which is crucial for
the proof of a parabolic Harnack inequality. For more detailed references
of Poincar\'e inequality, see Heinonen--Koskela~\cite{HeiKos98Qua} and 
Haj{\l}asz--Koskela~\cite{HajKos00Sob}. In particular, in the latter it has been
shown that if a weak $(1,2)$--Poincar\'e inequality is assumed, then
the Sobolev embedding theorem holds true and so a weak $(q,2)$--Poincar\'e inequality
holds for all $q\leq 2^*$, where, for a fixed exponent $p$ we have defined
\begin{equation}
p^\ast=\begin{cases}
\dfrac{pN}{N-p}, & p< N, \\
\mbox{ } & \\
+\infty, & p \ge N.
\end{cases}
\end{equation}

In addition, we have that if $u\in N^{1,2}_0(B_\rho)$, $B_\rho\subset \Omega$,
then the following Sobolev--type inequality is valid
\begin{equation}
\label{sobolev1}
\left( {{\int\!\!\!\!\!\!-}_{B_{\rho}}} |u|^{q} \,d\mu \right)^{1/q} \leq c_\ast \, 
\rho \, \left( {{\int\!\!\!\!\!\!-}_{B_{\rho}}} g_u^2 \,d\mu  \right)^{1/2},
\qquad  1\leq q\leq 2^*;
\end{equation}
for a proof of this fact we refer to Kinnunen--Shanmugalingam~\cite[Lemma 2.1]{kin-sha}.
The crucial fact here for us is that $2^\ast>2$. We also point out that since
$u\in N^{1,2}_0(B_\rho)$, then the balls in the previuous inequality have the same
radius. The fact that a weak $(1,p)$--Poincar\'e inequality
holds for $p>2-\varepsilon$ implies also the following Sobolev--type inequality
\begin{equation}
\label{sobolev2}
\left( {{\int\!\!\!\!\!\!-}_{B_{\rho}}} |u|^{q} \,d\mu \right)^{1/q} \leq C_p \, 
\rho \, \left( {{\int\!\!\!\!\!\!-}_{B_{\rho}}} g^p \,d\mu  \right)^{1/p},
\qquad  1\leq q\leq p^\ast,
\end{equation}
for any function $u$ with zero boundary values and any $g$ upper
gradient of $u$. The constant $c_\ast$ depends only on $c_d$ and on the constants
in the weak $(1,2)$--Poincar\'e inequality.

We also point out that requiring a Poincar\'e inequality implies
in particular the existence of ``enough'' rectifiable curves; this
also implies that the continuous embedding $N^{1,2} \to L^2$, given
by the identity map, is not onto. 

We now state and prove some results that are needed in the paper; these results
are stated for functions in $N^{1,2}$, but can be easily generalized to any
$N^{1,p}$, $1\leq p<+\infty$ if we assumed instead a weak $(1,p)$--Poincar\'e inequality.

\bthm
\label{sob-poin-bis}
Assume $u\in N_0^{1,2}(B_\rho)$, $0<\rho< \diam(X)/3$ (any $\rho>0$ in case $X$ 
has infinite diameter); then there
exist $\esp > 1$ such that we have
\[
{\int\!\!\!\!\!\!-}_{B_\rho} |u|^{2\esp} \, d\mu \leq c_\ast^2\rho^2
\left( {\int\!\!\!\!\!\!-}_{B_\rho} |u|^2\, d\mu \right)^{\esp-1}
	 {\int\!\!\!\!\!\!-}_{B_{\Lambda\rho}} g_u^2\,d\mu.
\]
\ethm

\begin{proof}
  Let $\esp=2-2/2^\ast$, where $2^\ast$ is as in the Sobolev inequality
  \eqref{sobolev1}. By H\"older's inequality and \eqref{sobolev1}, we
  obtain the claim
\[
\begin{split}
{\int\!\!\!\!\!\!-}_{B_{\rho}}|u|^{2\esp} \,d\mu
&\le\left( {\int\!\!\!\!\!\!-}_{B_{\rho}}|u|^2 \,d\mu \right)^{\esp-1}
\left( {\int\!\!\!\!\!\!-}_{B_{\rho}}|u|^{2^\ast} \,d\mu \right)^{2/2^\ast}
\\
&\le c_\ast^2\rho^2\left( {\int\!\!\!\!\!\!-}_{B_{\rho}}|u|^2 \,d\mu \right)^{\esp-1}
{\int\!\!\!\!\!\!-}_{B_{\Lambda\rho}} g_u^2\, d\mu.
\end{split}
\]
\end{proof}

By integrating the previous inequality in time, we obtain a parabolic
Sobolev inequality.

\bprop
\label{cor-sob-poin-bis}
Assume $u \in C([s_1,s_2]; L^2(X)) \cap L^2(s_1,s_2;
N_0^{1,2}(B_\rho))$. Then there exists $\esp > 1$ such that
\[
\int_{s_1}^{s_2}\!\!\! {{\int\!\!\!\!\!\!-}_{B_{\rho}}} |u|^{2\esp} \, d\mu \, dt
\leq c_\ast^2\rho^2  \left( \sup_{t\in(s_1,s_2)} {{\int\!\!\!\!\!\!-}_{B_{\rho}}}
|u(x,t)|^{2} \, d\mu(x) \right)^{\esp-1}\int_{s_1}^{s_2}\!\!\! {{\int\!\!\!\!\!\!-}_{B_{\rho}}}
g_u^2 \, d\mu \, dt.
\]
\eprop

We shall also need the following De Giorgi-type lemma.

\begin{lemma}
\label{lemma2.2}
Let $p>2-\varepsilon$ and $1\leq q \leq p^\ast$; moreover let $k, l
\in \R$ with $k < l$, and $u \in N^{1,2}(B_{\rho})$. Then
$$
(l-k)  \mu(\{u \leq k \} \cap B_\rho)^{1/q}  \mu (\{ u > l \}
\cap B_{\rho})^{1/q} \leq
 2C_p\rho  \mu(B_\rho)^{2/q-1/p}
 \left( \int_{\{ k < u < l \}\cap B_{\Lambda\rho}} g_u^p \, d\mu \right)^{1/p}.
$$
\end{lemma}
\boss
The previous result holds in every open set $\Omega \subset X$, provided that
\eqref{sobolev1} holds with $\Omega$ in place of $B_{\rho}$.
\eoss

\begin{proof}
Denote $A=\{x \in B_\rho: u(x) \leq k \}$; if $\mu(A)=0$, the result is immediate, otherwise, if $\mu(A)>0$, 
we define
$$
v:=\left\{
\ba{ll}
\min \{u , l \} - k,	&	\textrm{if } u > k,	\\
0,			&	\textrm{if } u \leq k .
\ea
\right.
$$
We have that
$$
\int_{B_\rho} |v - v_{B_\rho}|^q \, d\mu =
\int_{B_\rho\setminus A} |v - v_{B_\rho}|^q \, d\mu + \int_{A} |v_{B_\rho}|^q \, d\mu
\geq |v_{B_\rho}|^q \mu (A)
$$
and consequently
\begin{equation}
\label{misuradiA}
|v_{B_\rho}|^q \leq \frac{1}{\mu (A)}  \int_{B_\rho} |v - v_{B_\rho}|^q \, d\mu.
\end{equation}
On the other hand, we see that
\begin{equation}
\begin{split}
\int_{B_\rho} |v|^q \, d\mu
&= \int_{\{ u > l \}\cap B_\rho} (l-k)^q \, d\mu +
\int_{\{ k < u \leq l \}\cap B_\rho} |v|^q \, d\mu
\\
&\geq  (l-k)^q \mu(\{ u > l \}\cap B_\rho),
\end{split}
\end{equation}
and using \eqref{misuradiA}, we obtain
\[
\begin{split}
\left(\int_{B_\rho} |v|^q \, d\mu \right)^{1/q}
&\leq\left( \int_{B_\rho} |v - v_{B_\rho}|^q \, d\mu \right)^{1/q} +
\left( |v_{B_\rho}|^q \mu(B_\rho)\right)^{1/q}
\\
&\leq 2  \left( \frac{\mu(B_\rho)}{\mu(A)} \int_{B_\rho} |v - v_{B_\rho}|^q \,
d\mu \right)^{1/q} .
\end{split}
\]
By \eqref{sobolev2} and the doubling property, we finally conclude that
\[
(l-k) \mu (\{ u > l \} \cap B_{\rho})^{1/q}\leq
2C_p\rho \frac{\mu(B_\rho)^{2/q-1/p}} {\mu(A)^{1/q}}
	\left( \int_{B_{\Lambda\rho}} g_v^p  \, d\mu \right)^{1/p}  ,
\]
which is the required inequality.
\end{proof}

The following measure-theoretic lemma is a generalization of a result
obtained in \cite{DibGiaVes06Loc} to the metric space setting. Roughly
speaking, the lemma states that if the set where $u\in N^{1,1}_{\rm
  loc}(X)$ is bounded away from zero occupies a good piece of the ball
$B$, then there exists at least one point and a neighborhood about
this point such that $u$ remains large in a large portion of the
neighborhood. In other words, the set where $u$ is positive clusters
about at least one point of the ball $B$.

\begin{lemma}\label{lemmaMisVar}
Let $x_0\in X$, $\rho_0 > \rho>0$ with $\mu (\partial
B_{\rho}(x_0))=0$ and $\alpha,\beta > 0$. Then, for every
$\lambda,\delta\in (0,1)$ there exists $\eta\in (0,1)$ such that for
every $u\in N^{1,2}_{\rm loc}(X)$ satisfying
$$
\int_{B_{\rho_0}(x_0)} g^2_u \,d\mu \leq \beta  \frac{\mu(B_{\rho_0}(x_0))}{\rho^2_0},
$$
and
$$
\mu(\{ u>1\} \cap B_\rho(x_0))\geq \alpha \mu(B_\rho(x_0)),
$$
there exists $x^\ast\in B_\rho(x_0)$ with $B_{\eta \rho}(x^*)\subset B_\rho(x_0)$ and
$$
\mu(\{ u>\lambda \}\cap B_{\eta\rho}(x^\ast))>(1-\delta)\mu(B_{\eta\rho}(x^\ast)).
$$
\end{lemma}

\boss 
The assumption $\mu(\partial B_{\rho}(x_0))=0$ is not
restrictive, since this property holds except for at most countably many
radii $\rho>0$ and we can choose the appropriate radius $\rho$ as we
like. We also point out that the two previous lemmas can also be
stated for functions of bounded variation instead of Sobolev
functions, once a weak $(1,1)$--Poincar\'e inequality is assumed; 
the proofs given here can be easily adapted to this case by
using the notion of the perimeter.  
\eoss

\begin{proof}
For every $\eta<(\rho_0-\rho)/(2\Lambda\rho)$, we may consider a
finite family of disjoint balls $\{B_{\eta\rho}(x_i)\}_{i\in I}$,
$x_i\in B_{\rho}(x_0)$ for every $i\in I$, $B_{\eta\rho}(x_i)\subset B_\rho(x_0)$, 
such that
$$
B_{\rho}(x_0) \subset \bigcup_{i\in I}B_{2\eta\rho}(x_i)
\subset B_{\rho_0}(x_0) .
$$
Observe that $B_{2\Lambda\eta\rho}(x_i)\subset B_{\rho_0}(x_0)$ for every $i\in I$
and by the doubling property, the balls $B_{2\Lambda\eta\rho}(x_i)$
have bounded overlap with bound independent of $\eta$.
We denote
$$
I^+=\left\{i\in I: \mu(\{u>1\} \cap B_{2\eta\rho}(x_i))>\frac{\alpha}{2c_d}\mu(B_{2\eta\rho}(x_i))\right\}
$$
and
$$
I^-=\left\{i\in I: \mu(\{u>1\} \cap B_{2\eta\rho}(x_i))\leq \frac{\alpha}{2c_d}
\mu(B_{2\eta\rho}(x_i))\right\}.
$$
By assumption, we get
\begin{align*}
\alpha \mu(B_\rho(x_0))&
\leq \mu(\{u>1\}\cap B_\rho(x_0)) \\
&\leq
\sum_{i\in I^+}\mu(\{u>1\}\cap B_{2\eta\rho}(x_i))
+\frac{\alpha}{2c_d} \sum_{i\in I^-}\mu(B_{2\eta\rho}(x_i)) \\
&\leq
\sum_{i\in I^+}\mu(\{u>1\}\cap B_{2\eta\rho}(x_i))
+\frac{\alpha}{2} \sum_{i\in I^-}\mu(B_{\eta\rho}(x_i)) \\
&\leq
\sum_{i\in I^+}\mu(\{u>1\}\cap B_{2\eta\rho}(x_i))
+\frac{\alpha}{2}\mu(B_{(1+\eta)\rho}(x_0))
\end{align*}
and consequently
\begin{equation}\label{esti1}
\frac{\alpha}{2}\left(
\mu(B_\rho(x_0))-\mu(B_{(1+\eta)\rho}(x_0)\setminus B_\rho(x_0))
\right)
\leq
\sum_{i\in I^+}\mu(\{u>1\}\cap B_{2\eta\rho}(x_i)).
\end{equation}
Assume by contradiction that
\begin{equation}\label{estDelta}
\mu(\{u>\lambda\}\cap B_{\eta\rho}(x_i))\le(1-\delta)\mu(B_{\eta\rho}(x_i)),
\end{equation}
for every $i\in I^+$;
this clearly implies that
$$
\frac{\mu(\{u\leq \lambda\}\cap B_{\eta\rho}(x_i))}{\mu(B_{\eta\rho}(x_i))}\geq \delta.
$$
The doubling condition on $\mu$ also implies that
$$
\frac{\mu(\{u\leq \lambda\}\cap B_{2\eta\rho}(x_i))}{\mu(B_{2\eta\rho}(x_i))}\geq \frac{\delta}{c_d}.
$$
By Lemma \ref{lemma2.2} with $q=2$, $k=\lambda$ and $l=1$,
we obtain that
\begin{align}\label{estOmega}
\nonumber
\frac{\delta}{c_d}\mu(\{u>1\}\cap B_{2\eta\rho}(x_i)) \leq &
\frac{\mu(\{u\leq \lambda\}\cap B_{2\eta\rho}(x_i))}{\mu(B_{2\eta\rho}(x_i))}
\mu(\{u>1\}\cap B_{2\eta\rho}(x_i)) \\
\leq &
\frac{16C_2^2\eta^2 \rho^2}{(1-\lambda)^2} \int_{\{\lambda<u< 1\}\cap B_{2\Lambda\eta\rho}(x_i)}g_u^2 \,d\mu.
\end{align}
Summing up over $I^+$ and using the bounded overlapping property, from \eqref{esti1} we get
\begin{align*}
\frac{\alpha}{2}(1-\lambda)^2\frac{\delta}{c_d}&\left(
\mu(B_\rho(x_0))-\mu(B_{(1+\eta)\rho}(x_0)\setminus B_\rho(x_0) \right) \\
& \leq 16C_2^2  \eta^2  \rho^2  \sum_{i\in I^+}
\int_{\{\lambda \leq u < 1\}\cap B_{2\Lambda\eta\rho}(x_i)}g^2_u\,d\mu \\
&\leq c' \eta^2 \rho^2 \int_{B_{\rho_0}(x_0)}g^2_u\,d\mu\\
&\leq c'\beta  \mu(B_{\rho_0}(x_0)) \eta^2,
\end{align*}
where the costant $c'$ is given by $16C_2^2$ multiplied by the overlapping constant.
The conclusion follows by passing to the limit with $\eta\to 0$, since
the condition $\mu(\partial B_\rho(x_0))=0$ implies that the left hand side
of the previous equation tends to
$$
\frac{\alpha}{2}(1-\lambda)^2\frac{\delta}{c_d}\mu(B_\rho(x_0)).
$$
\end{proof}

We conclude with a result which will be needed later; for the proof we
refer, for instance, to \cite[Lemma 7.1]{giusti}.

\begin{lemma}
\label{lemmuzzo}
Let $\{y_h\}_{h=0}^\infty$ be  a sequence of positive real numbers such that
\[
y_{h+1} \leq c  b^h y_h^{1+\alpha},
\]
where $c>0$, $b>1$ and $\alpha > 0$. Then if $y_0 \leq c^{-1/\alpha} b
^{-1/\alpha^2}$, we have
\[
\lim_{h\to \infty} y_h = 0.
\]
\end{lemma}

\section{Parabolic De Giorgi classes and quasiminimizers}
\label{secDgQm}

We consider a variational approach related to the heat
equation (see Definition \ref{def_of_sol})
\begin{equation}
\label{heatequation}
\frac{\partial u}{\partial t} - \Delta u = 0 \hskip30pt \textrm{in } \Omega \times (0,T)
\end{equation}
and provide a Harnack inequality for a class of functions in a metric
measure space generalizing the known result for positive solutions of
\eqref{heatequation} in the Euclidean case. The following definition is essentially
based on the approach of DiBenedetto--Gianazza--Vespri~\cite{DibGiaVes06Loc} and also of Wieser~\cite{wieser}; 
we refer also to the book of Lieberman~\cite{Lieberman} for a more detailed description.

\bd[Parabolic De Giorgi classes of order 2]
\label{classiDG}
Let $\Omega$ be a non-empty open subset of $X$ and $T>0$. A function
$u : \Omega \times (0,T) \to \R$ belongs to the class $DG_+(\Omega, T,
\gamma)$, if
\[
u \in C([0,T]; L^2_{\rm loc}(\Omega)) \cap L^2_{\rm loc}(0,T; N^{1,2}_{\rm loc}(\Omega)),
\]
and for all $k\in\R$ the following energy estimate holds 
\begin{align}
\label{DGgamma}
\sup_{t \in (\tau,s_2)} & \int_{B_{r}(x_0)} (u-k)_+^2 (x,t) \, d \mu +
\int_{\tau}^{s_2}\!\!\!\!  \int_{B_{r}(x_0)} g_{(u-k)_+}^2 \, d\mu\,
ds \leq \alpha \int_{B_{R}(x_0)} (u-k)_+^2 (x,s_1) \, d \mu(x) \\
& \qquad +
\gamma\left( 1 + \frac{1-\alpha}{\theta}\right) \frac1{(R-r)^2}
\int_{s_1}^{s_2}\!\!\!\!\int_{B_{R}(x_0)} (u-k)_+^2 \,d\mu\, ds,
\nonumber
\end{align}
where $(x_0,t_0) \in \Omega \times (0,T)$, and $\theta>0$, $0<r<R$,
$\alpha \in [0,1]$, $s_1, s_2 \in (0,T)$, and $s_1 < s_2$ are so that
\[
\tau, t_0 \in [s_1,s_2],
\quad s_2 - s_1 = \theta R^2,
\quad\tau - s_1 = \theta (R-r)^2,
\]
and $B_R(x_0) \times (t_0 - \theta R^2, t_0 + \theta R^2) \subset
\Omega \times (0,T)$. The function $u$ belongs to $DG_-(\Omega, T,
\gamma)$ if \eqref{DGgamma} holds with $(u-k)_+$ replaced by
$(u-k)_-$.
The function $u$ is said to belong to the parabolic De Giorgi
class of order 2, denoted $DG(\Omega, T, \gamma)$, if
\[
u \in DG_+(\Omega, T, \gamma) \cap DG_-(\Omega, T, \gamma).
\]
\ed

In what follows, the estimate \eqref{DGgamma} given in
Definition~\ref{classiDG} is referred to as \emph{energy estimate} or
\emph{Caccioppoli-type estimate}.  We also point out that our
definition of parabolic De Giorgi classes of order 2 is sligthly
different from that given in the Euclidean case by
Gianazza--Vespri~\cite{GiaVes06Par}; our classes seem to be larger,
but it is not known to us whether they are equivalent.

Denote $\mathcal{K}(\Omega \times (0,T))=\{ K \subset \Omega
\times (0,T): K \textrm{ compact} \}$ and consider the functional
$$
E: L^2(0,T; N^{1,2}(\Omega)) \times \mathcal{K}(\Omega \times (0,T)) \to \R, \qquad
E(w, K) = \frac{1}{2}  \iint_K g_w^2 \,d\mu\,dt.
$$

\bd[Parabolic quasiminimizer] \label{def:parabolicQM} Let $\Omega$ be
an open subset of $X$.  A function 
$$
u \in L^2_{\rm loc}(0,T;N^{1,2}_{\rm loc}(\Omega))
$$ 
is said to be a parabolic
$Q$-quasiminimizer, $Q\geq 1$, related to the heat equation
\eqref{heatequation} if \be \label{Qminpar} -\iint_{\supp(\phi)} u
\frac{\partial \phi}{\partial t} \,d\mu\,dt + E(u, \supp(\phi)) \leq
QE(u-\phi,\supp(\phi)) \ee for every $\phi \in$ $\Lip_c(\Omega \times
(0,T))=\{f\in\Lip(\Omega \times (0,T)): \supp(f)\Subset \Omega \times
(0,T)\}$.  \ed

In the Euclidean case with the Lebesgue measure it can be shown that
$u$ is a weak solution of \eqref{heatequation} if and only if $u$ is a
$1$-quasiminimizer for \eqref{heatequation}, see \cite{wieser}.  Hence
$1$-quasiminimizers can be seen as weak solutions of
\eqref{heatequation} in metric measure spaces.  This motivates the
following definition.

\bd
\label{def_of_sol}
A function $u$ is a parabolic minimizer if $u$ is a parabolic $Q$-quasiminimizer
with $Q=1$.  \ed

We also point out that the class of $Q$-quasiminimizers is non-empty
and non-trivial, since it contains the elliptic $Q$-quasiminimizers
as defined in \cite{GiaGiu82OnT,GiaGiu84Qua} and as shown there, there
exist many other examples as well.

\boss
It is possible to prove, by using the Cheeger differentiable structure and
the same proof contained in Wieser \cite[Section 4]{wieser}, that 
a parabolic $Q$--quasiminimizer belongs to a suitable parabolic De~Giorgi class. We are
not able to prove this result directly without using the Cheeger differentiable 
structure; the main problem is that the map $u\mapsto g_u$ is only 
sublinear and not linear, and linearity is a main tool used in the argument.
\eoss

\section{Parabolic De Giorgi classes and local boundedness}
\label{secLocBdd}

We shall use the following notation;
$$
Q_{\rho, \theta}^+(x_0,t_0) = B_{\rho}(x_0) \times [t_0, t_0
  + \theta \rho^2),
$$
$$
Q_{\rho, \theta}^-(x_0,t_0) = B_{\rho}(x_0) \times (t_0 -
  \theta \rho^2, t_0],
$$
$$
Q_{\rho, \theta}(x_0,t_0) = B_{\rho}(x_0) \times (t_0 -
  \theta \rho^2, t_0 + \theta \rho^2).
$$ 
When $\theta=1$ we shall simplify the notation by writing $Q_{\rho}^+(x_0,t_0) =Q_{\rho, 1}^+(x_0,t_0)$,
$Q_{\rho}^-(x_0,t_0) =Q_{\rho, 1}^-(x_0,t_0)$ and $Q_{\rho} (x_0,t_0) =Q_{\rho, 1}(x_0,t_0)$.

We shall show that functions belonging to $DG(\Omega,T,\gamma)$
are locally bounded.  Here we follow the analogous proof contained in
\cite{kin-sha} for the elliptic case.  Consider $r, R > 0$ such that
$R/2 < r < R$, $s_1, s_2 \in (0,T)$ with $2(s_2 - s_1) = R^2$ and
$\sigma \in (s_1, s_2)$ such that $\sigma < (s_1+s_2)/2$, fix $x_0 \in
X$.  We define level sets at scale $\rho>0$ as follows
\[
A(k;\rho; t_1,t_2) := \{ (x,t) \in B_{\rho}(x_0) \times (t_1,t_2): u
(x,t) > k \}.
\]
Let $\tilde r:= (R+r)/2$, i.e., $R/2 < r < \tilde
r < R$, and $\eta\in \Lip_c(B_{\tilde r})$ such that $0\leq\eta\leq
1$, $\eta=1$ on $B_r$, and $g_\eta \leq 2/(R-r)$. Then
$v=(u-k)_+\eta\in N_0^{1,2}(B_{\tilde r})$ and $g_v\leq g_{(u-k)_+} +
2(u-k)_+/(R-r)$. We have
\begin{align*}
& \mediad{B_{r}\times(\sigma,s_2)}(u-k)_+^2\,d\mu\,dt \leq 2^N\mediad{B_{\tilde r}\times(\sigma,s_2)}(u-k)_+^2\eta^2\,d\mu\,dt \\
  & \leq \frac{2^N}{\mis (B_{\tilde r} \times (\sigma, s_2))}
  \left(\int\!\!\!\!\int_{B_{\tilde r}\times (\sigma,s_2)} (u-k)_+^q\eta^q
    \,d\mu\,dt \right)^{2/q}
  \left( \mis (A(k; \tilde r; \sigma, s_2)\right)^{(q-2)/q}		\\
  & \leq 2^N\left( \frac{\mis (A(k;\tilde r;\sigma,s_2))}{\mis (B_{\tilde r}
      \times (\sigma,s_2))} \right)^{(q-2)/q}
  \left(\mediad{B_{\tilde r}\times (\sigma,s_2)} (u-k)_+^q\eta^q\, d\mu\, dt
  \right)^{2/q} .
\end{align*}
We now use Proposition \ref{cor-sob-poin-bis} taking $q=2\esp$. We get  
\begin{align*}
\mediad{B_r\times (\sigma,s_2)} & (u-k)_+^2\,d\mu\, dt \leq 2^N2^{2/\esp}
  c_{\ast}^{2/\esp}r^{2/\esp} \left(\frac{\mis
      (A(k;\tilde r;\sigma,s_2))}{\mis (B_{\tilde r} \times (\sigma,s_2))}
  \right)^{(\esp-1)/\esp} \\
  & \times \left( \sup_{t\in(\sigma,s_2)}
    {{\int\!\!\!\!\!\!-}_{B_{\tilde r}}}(u-k)_+^2\eta^2\, d\mu
    \right)^{(\esp-1)/\esp} \left(\mediad{B_{\tilde r}\times (\sigma,s_2)}
      g_{(u-k)_+\eta}^2 \,d\mu\,dt\right)^{1/\esp}.
\end{align*}
By applying \eqref{DGgamma} with $\tau = \sigma$, $\alpha = 0$, and
$\theta = 1/2$, since $\esp>1$, we arrive at
\begin{align*}
  \mediad{B_{r}\times(\sigma,s_2)} & (u-k)_+^2 \,d\mu\,dt \leq
  2^{N+2}\frac{c_{\ast}^{2/\esp}r^{2/\esp}}{(s_2-\sigma)^{1/\esp}} \left(
    \frac{\mis (A(k;\tilde r;\sigma,s_2))}{\mis (B_{\tilde r} \times
      (\sigma,s_2))}
  \right)^{(\esp-1)/\esp} \\
  & \quad \times \left(\sup_{t\in(\sigma,s_2)}
    {{\int\!\!\!\!\!\!-}_{B_{\tilde r}}} (u-k)_+^2(x,t)\, d\mu
  \right)^{(\esp-1)/\esp}
  \left(2\int_{\sigma}^{s_2}\!\!\! {{\int\!\!\!\!\!\!-}_{B_{\tilde r}}} g_{(u-k)_+}^2\,d\mu\,dt\right. \\
  & \qquad\qquad + \left. \frac{8}{(R-r)^2}\int_{\sigma}^{s_2}\!\!\! {{\int\!\!\!\!\!\!-}_{B_{\tilde r}}}(u-k)_+^2\,d\mu\,dt\right)^{1/\esp} \\
  \leq & 2^{N+2}\frac{c_{\ast}^{2/\esp} r^{2/\esp}}{(s_2-\sigma)^{1/\esp}}
  \left( \frac{\mis (A(k;\tilde r;\sigma,s_2))}{\mis (B_{\tilde r} \times
      (\sigma,s_2))}
  \right)^{(\esp-1)/\esp}\frac{\mu(B_R)}{\mu(B_{\tilde r})} \\
  & \qquad \times \frac{6\gamma+2^{2N+3}}{(R-r)^2} \int_{s_1}^{s_2}\!\!\!
  {{\int\!\!\!\!\!\!-}_{B_{R}}} (u-k)_+^2 \,d\mu\,dt .
\end{align*}
By the the choice of $\sigma$, we see that $(s_2 - \sigma)^{-1} < 2 (s_2-s_1)^{-1}$,
and consequently
\begin{align} \label{bounded1}
\mediad{B_{r}\times (\sigma,s_2)} & (u-k)_+^2\,d\mu\,dt
\leq 2^{N+4}(3\gamma+2^{2N+2})c_{\ast}^{2/\esp} r^{2/\esp}
\left( \frac{\mis (A(k;\tilde r;\sigma,s_2))}{\mis (B_{\tilde r} \times
(\sigma,s_2))} \right)^{(\esp-1)/\esp} \\
& \qquad \times \frac{\mu(B_R)}{\mu(B_{\tilde r})}(s_2-\sigma)^{(\esp-1)/\esp}
\frac{1}{(R-r)^2} \mediad{B_{R}\times (s_1,s_2)} (u-k)_+^2 \,d\mu\,dt \nonumber
\end{align}
Consider $h < k$. Then
\begin{align*}
(k-h)^2 & \left( \mis (A(k;\tilde r;\sigma,s_2))\right) \leq
	\int\!\!\!\!\int_{A(k;\tilde r;\sigma,s_2)} (u-h)_+^2\,d\mu\,dt \\
& \leq \int\!\!\!\!\int_{A(h;\tilde r;\sigma,s_2)} (u-h)_+^2  \,d\mu\,dt =
	\int_{\sigma}^{s_2}\!\!\!\!\int_{B_{\tilde r}} (u-h)_+^2 \,d\mu\,dt,
\end{align*}
from which, using the doubling property \eqref{stime_doub_e_revdoub},
it follows that
\begin{align} \label{bounded2}
  \mis (A(k;\tilde r;\sigma,s_2)) & \leq \frac{1}{(k-h)^2} \left(\mis\left(
      B_{\tilde r} \times (\sigma, s_2)
    \right) \right) u(h;\tilde r;\sigma,s_2)^2 \\
  & \leq \frac{2^{N+1}}{(k-h)^2} \left(\mis\left( B_{\tilde r} \times
      (\sigma, s_2) \right) \right) u(h;R;s_1,s_2)^2, \nonumber,
\end{align}
where 
\[
u(l;\rho;t_1,t_2) := \left( \mediad{B_{\rho}\times (t_1,t_2)}
(u-l)_+^2 d\,\mu\,dt\right)^{1/2}.
\]
By plugging \eqref{bounded2} into \eqref{bounded1} and arranging terms
we arrive at
\begin{equation} \label{l'unicachecito}
u(k;r;\sigma,s_2)\leq
\bar c\frac{r^{1/\esp}(s_2-\sigma)^{(\esp-1)/2\esp}}{(k-h)^{(\esp-1)/\esp}(R-r)}u(k;R;s_1,s_2)
u(h;R;s_1,s_2)^{(\esp-1)/\esp},
\end{equation}
with $\bar c = 2^{N+2+(N+1)(\esp-1)/(2\esp)}(3\gamma+2^{2N+2})^{1/2}c_{\ast}^{1/\esp}$. 

Let us consider the following sequences: for $n \in \N$, $k_0\in \R$
and fixed $d$ we define
\begin{align*}
  & k_n := k_0 + d \left( 1 - \frac{1}{2^n} \right) \nearrow k_0 + d, \\
  & r_n := \frac{R}{2} + \frac{R}{2^{n+1}}		\searrow \frac{R}{2},\quad\textrm{and}	\\
  & \sigma_n := \frac{s_1+s_2}{2} - \frac{R^2}{4^{n+1}} \nearrow
  \frac{s_1+s_2}{2}.
\end{align*}
This is possible since $2(s_2 - s_1) = R^2$. The following technical
result will be useful for us.

\begin{lemma} \label{lemmuzzoinduttivo}
Let $u_0 := u(k_0; R; s_1, s_2)$, $u_n := u(k_n;r_n;\sigma_n,s_2)$,
\[
\theta := \frac{\esp-1}\esp,\quad a := \frac{1+\theta}\theta = 1 +
\frac\esp{\esp-1},
\]
and
\[
d^{\theta} = 
\bar c\,2^{1+\theta/2 + a(1+\theta)}u_0^\theta,
\]
where $\bar c$ is the constant in \eqref{l'unicachecito}. Then
\begin{equation}
\label{decadimento}
u_{n} \leq \frac{u_0}{2^{an}}.
\end{equation}
\end{lemma}

\begin{proof}
We prove the lemma by induction.
First notice that \eqref{decadimento} is true for $n=0$.
Then assume that \eqref{decadimento} is true for fixed $n \in \N$.
In \eqref{l'unicachecito}, we first estimate
$r^{1/\esp}(s_2-\sigma)^{(\esp-1)/2\esp}$ by
$R^{1/\esp}(s_2-s_1)^{(\esp-1)/2\esp}$.
Then we replace $r$ with $r_{n+1}$, $R$ with $r_n$,
$\sigma$ with $\sigma_{n+1}$, $s_1$ with $\sigma_{n}$,
$h$ with $k_n$, and $k$ with $k_{n+1}$.
With these replacements we arrive at
\[
u_{n+1} \leq \frac{\bar c\,R^{1/\esp}(s_2-s_1)^{(\esp-1)/(2\esp) }}
{ (k_{n+1}-k_n)^{(\esp-1)/\esp}(r_{n}-r_{n+1}) }
u_n^{1 + (\esp-1)/\esp} .
\]
Denote $c' := \bar c\, R^{1/\esp}(s_2-s_1)^{(\esp-1)/(2\esp)}$ so that
we have
\[
u_{n+1} \leq \frac{c'u_n^{1 + \theta}}{(k_{n+1}-k_n)^{\theta}(r_{n}-r_{n+1})}.
\]
Since $r_n - r_{n+1} = 2^{-(n+2)}R$ and
$k_{n+1} - k_n = 2^{-(n+1)} d$, we obtain
\[
u_{n+1} \leq c'\frac{2^{(n+1)\theta+n+2}}{d^{\theta}R}u_n^{1 + \theta}
	= 2c'\frac{2^{(n+1)(1+\theta)}}{d^{\theta}R}
	u_n^{1 + \theta}
	\leq 2c'\frac{2^{(n+1)(1+\theta)}}{d^{\theta}R}
	 \left(\frac{u_0}{2^{an}}\right)^{1 + \theta}.
\]
As $2(s_2-s_1)=R^2$, we have
\begin{equation*}
c'' : = \frac{2c'}{R} = 2\bar c\, R^{1/\esp}(s_2-s_1)^{(\esp-1)/(2\esp)}\frac{1}{R}
= 2^{1+\theta/2}\bar c.
\end{equation*}
Point being that the constant $c''$ is independent of $R$, $s_1$, and $s_2$.  Finally, since $(1-a)(1+\theta) = - a$ we arrive at
\[
u_{n+1} \leq c''\frac{2^{(n+1)(1+\theta)}}{d^{\theta}}\left(\frac{u_0}{2^{an}} \right)^{1 + \theta} = 2^{-a (n+1)} u_0 .
\]
This completes the proof.
\end{proof}

Before proving the main result of this section, we need the following
proposition.

\begin{prop}\label{Prop4.2}
For every number $k_0 \in \R$ we have
\[
u(k_0 + d; R/2; (s_1+s_2)/2, s_2) = 0,
\]
where $d$ is defined as in Lemma~\ref{lemmuzzoinduttivo}.
\end{prop}

\begin{proof}
Since $k_n \leq k_0 + d$, $R/2 \leq r_n \leq R$,
$s_1 \leq \sigma_n \leq (s_1+s_2)/2$,
the doubling property implies that
\[
0 \leq u(k_0 + d; R/2; (s_1+s_2)/2, s_2) \leq 2^{N+1}
u(k_n; r_n; \sigma_n, s_2)=u_n.
\]
By Lemma \ref{lemmuzzoinduttivo}, we have $\lim_{n\to\infty} u_n = 0$ and the
claim follows.
\end{proof}

We close this section by proving local boundedness for functions in the De
Giorgi class.

\begin{theorem}
\label{Linfinity}
Suppose $u \in DG(\Omega,T,\gamma)$. Then there is a constant $c_\infty$
depending only on $c_d$, $\gamma$, and the constants in the weak $(1,2)$--Poincar\'e 
inequality, such that for all $B_{R} \times
(s_1,s_2) \subset \Omega \times (0,T)$, we have
\[
\esssup_{B_{R/2} \times ((s_1+s_2)/2, s_2)} |u|
\leq c_\infty \left( \mediad{B_{R}\times (s_1,s_2)} |u|^2 d \mu\, dt \right)^{1/2} .
\]
\end{theorem}

\begin{proof}
The Proposition \ref{Prop4.2} implies that
\[
\esssup_{B_{R/2} \times ((s_1+s_2)/2, s_2)} u \leq k_0 + d,
\]
where $d$ is defined in Lemma \ref{lemmuzzoinduttivo}. Then
\[
\esssup_{B_{R/2} \times ((s_1+s_2)/2, s_2)} u \leq k_0 +
	c_\infty \left( \mediad{B_{R}\times (s_1,s_2)} (u-k_0)_+^2 d \mu\, dt \right)^{1/2},
\]
with $c_\infty=\bar c^{1/\theta}\, 2^{1+\theta/2+a(1+\theta)}$,  $\bar c$ the constant in \eqref{l'unicachecito}.
The previous inequality with $k_0 = 0$ can be written as follows
\[
\esssup_{B_{R/2} \times ((s_1+s_2)/2, s_2)} u \leq
	c_\infty \left( \mediad{B_{R}\times (s_1,s_2)} u_+^2 d \mu\, dt \right)^{1/2}
	\leq c_\infty \left( \mediad{B_{R}\times (s_1,s_2)} |u|^2 d \mu\, dt \right)^{1/2}.
\]
Since also $-u \in DG(\Omega,T,\gamma)$ the analogous
argument applied to $-u$ gives the claim.
\end{proof}

\section{Parabolic De Giorgi classes and Harnack inequality}
\label{secHarnack}

In this section we shall prove a scale-invariant parabolic Harnack inequality for functions in the De Giorgi
class of order 2 and, in particular, for parabolic quasiminimizers.

\bprop
\label{prop-DeGiorgi1}
Let $\rho,\,\theta > 0$ be
chosen such that the cylinder $Q_{\rho, \theta}^-(y,s)\subset\Omega
\times (0,T)$.  Then for each choice of $a, \sigma \in (0,1)$ and
$\bar\theta \in (0, {\theta})$, there is $\nu_+$, depending only on
$N, \gamma, c_{\ast}, a, \theta, \bar\theta$,
such that for every $u \in DG_+(\Omega, T, \gamma)$
and $m_+$ and $\om$ for which
$$
m_+ \geq \esssup_{Q_{\rho, \theta}^-(y,s)} u
\quad\textrm{and}\quad
\om \geq \osc_{Q_{\rho, \theta}^-(y,s)} u,
$$
the following claim holds true: if
\[
\mis \left( \{ (x,t) \in Q_{\rho, \theta}^-(y,s): u(x,t) > m_+ -\sigma \om \} \right)
	\leq \nu_+  \mis \left( Q_{\rho, \theta}^-(y,s) \right),
\]
then
$$
u(x,t) \leq m_+ - a  \sigma  \om \hskip30pt
\mis\text{-a.e. in } B_{\rho/2}(y)\times (s-\bar\theta \rho^2,s].
$$
\eprop

\begin{proof}
Define the following sequences, $h \in \N$,
\begin{align*}
\rho_h & := \frac{\rho}{2} + \frac{\rho}{2^{h+1}} \searrow \frac{\rho}{2}\, ,	\quad
	\theta_h = \bar{\theta} + \frac{1}{2^h} (\theta - \bar\theta) \searrow \bar\theta\, , 	\\
B_h & := B_{\rho_h}(y),	\quad 
 	s_h := s - \theta_h \rho^2 \nearrow s - \bar\theta\rho^2 \, , \quad
	Q_h^- := B_h \times (s_h, s],									\\
\sigma_h & := a \sigma + {\frac{1 - a}{2^h}} \sigma \searrow
  a\sigma,\quad \textrm{and}\quad k_h = m_+ - \sigma_h \om \nearrow
  m_+-a\sigma\om.
\end{align*}
Consider a sequence of Lipschitz continuous functions $\zeta_h$, $h \in \N$, satisfying the following:
\begin{align*}
& \zeta_h \equiv 1 \text{ in } Q_{h+1}^- \, , \qquad 
	\zeta_h \equiv 0 \text{ in } Q^-_{\rho,\theta}(y, s) \setminus {Q}^-_{h}				\\
& g_{\zeta_h} \leq \frac{1}{\rho_h - \rho_{h+1}} = \frac{2^{h+2}}{\rho} \, ,
	\hskip30pt 0 \leq 
	(\zeta_h)_t \leq \frac{2^{h+1}}{\theta - \bar\theta} \, \frac{1}{\rho^2}	\, .
\end{align*}
Denote $A_h := \{ (x,t) \in Q_h^- : u(x,t) > k_h \}$. We have
\begin{align*}
\int\!\!\!\!\int_{Q_{h}^-} & (u-k_h)_+^2 \zeta^2_h \, d\mu \, dt \geq
\int\!\!\!\!\int_{Q_{h+1}^-} (u-k_h)_+^2 \, d\mu \, dt \geq
\int\!\!\!\!\int_{A_{h+1}} (u-k_h)_+^2 \,d\mu\,dt \\
& \geq \int\!\!\!\!\int_{A_{h+1}}(k_{h+1}-k_h)^2 \,d\mu \, dt =
{\frac{((1-a)\sigma\om)^2}{2^{2h+2}}}\mis({A_{h+1}}),
\end{align*}
and consequently
\be \label{primo}
\int_{s_{h}}^s \!\! {{\int\!\!\!\!\!\!-}_{\!\!\!B_{h}}}
(u-k_h)_+^2 \zeta^2_h \, d\mu \, dt \geq \frac{((1-a) \sigma \om)^2}{2^{2h+2}}
\frac{\mis ({A_{h+1}})}{\mu (B_{h+1})}.
\ee
On the other hand, if we use first H\"older's inequality and then Proposition~\ref{cor-sob-poin-bis},
we obtain the following estimate
\begin{align*}
\int_{s_{h}}^s {{\int\!\!\!\!\!\!-}_{\!\!\!B_{h}}} (u - & k_h)_+^2 \zeta^2_h \,d\mu\,dt \leq 
	\left(\frac{\mis({A_{h}})}{\mu (B_h)}\right)^{(\esp-1)/\esp}
	\bigg( \int_{s_{h}}^s\!\!{{\int\!\!\!\!\!\!-}_{\!\!\!B_h}}
	(u-k_h)_+^{2\esp}\zeta_h^{2\esp}\,d\mu\,dt\bigg)^{1/\esp}						\\
\leq & \, c_{\ast}^{2/\esp} \rho^{2/\esp}\left( \frac{\mis ({A_{h}})}{\mu (B_h)}
	\right)^{(\esp-1)/\esp} \bigg(\sup_{t\in (s_{h},s)}
	{{\int\!\!\!\!\!\!-}_{\!\!\!B_h}}(u-k_h)_+^2 \zeta_h^2 \, d\mu \bigg)^{(\esp-1)/\esp} \times 		\\
     & \quad \times
	\bigg(\int_{s_{h}}^s\!\!{{\int\!\!\!\!\!\!-}_{\!\!\!B_h}} \left(2 \, \zeta_h^2 \, g_{(u-k_h)_+}^2 + 
	2 g_{\zeta_h}^2 (u-k_h)_+^2\right) \,d\mu\,dt \bigg)^{1/\esp}					\\
\leq & \,c_{\ast}^{2/\esp}\rho^{2/\esp}
	\left(\frac{\mis({A_{h}})}{\mu(B_h)}\right)^{(\esp-1)/\esp}
	\!\!\!\!\! \frac{1}{\mu(B_h)} \bigg(\sup_{t\in (s_{h},s)}
	{{\int}_{\!\!\!B_h}}(u-k_h)_+^2 \, d\mu \bigg)^{(\esp-1)/\esp}	\times					\\
& \quad \times \bigg( 2 \int_{s_{h}}^s\!\!{{\int}_{\!\!\!B_{h}}} \, g_{(u-k_h)_+}^2 \, d\mu\,dt
	+ \frac{2^{2h+5}}{\rho^2}
	\int_{s_{h}}^s\!\!{{\int}_{\!\!\!B_{h}}}(u-k_h)_+^2\,d\mu\,dt \bigg)^{1/\esp}  			\\
\leq & \, 2^{1/\esp} \, c_{\ast}^{2/\esp}\rho^{2/\esp}
	\left(\frac{\mis({A_{h}})}{\mu(B_h)}\right)^{(\esp-1)/\esp}
	\!\!\!\!\! \frac{1}{\mu(B_h)} \bigg(\sup_{t\in (s_{h},s)}
	{{\int}_{\!\!\!B_h}}(u-k_h)_+^2 \, d\mu \, +								\\
& \quad + \int_{s_{h}}^s\!\!{{\int}_{\!\!\!B_{h}}} \, g_{(u-k_h)_+}^2 \, d\mu\,dt
	+ \frac{2^{2h+4}}{\rho^2}
	\int_{s_{h}}^s\!\!{{\int}_{\!\!\!B_{h}}}(u-k_h)_+^2\,d\mu\,dt \bigg) \, .
\end{align*}
We continue by applying the energy estimate \eqref{DGgamma} with $r = \rho_h$, $R = \rho_{h-1}$,
$\alpha = 0$, $s_2 = s$, $\tau = s_h$, $s_1 = s_{h-1}$ and get
\begin{align*}
\int_{s_{h}}^s {{\int\!\!\!\!\!\!-}_{\!\!\!B_{h}}} (u - & k_h)_+^2 \zeta^2_h \,d\mu\,dt \\
\leq & \, 2^{1/\esp} \, c_{\ast}^{2/\esp} \rho^{2/\esp}
	\left(\frac{\mis({A_{h}})}{\mu(B_h)}\right)^{(\esp-1)/\esp}
	\!\!\!\!\! \frac{1}{\mu(B_h)} \bigg(\frac{2^{2h+4}}{\rho^2}
	\int_{s_{h}}^s\!\!{{\int}_{\!\!\!B_{h}}}(u-k_h)_+^2\,d\mu\,dt +						\\
& \hskip20pt + \gamma \left(1 + \frac{2^{h}}{\theta - \bar\theta} \right) \frac{2^{2h+2}}{\rho^2}
	\int_{s_{h-1}}^s\!\!{{\int}_{\!\!\!B_{h-1}}}(u-k_h)_+^2\,d\mu\,dt \bigg)				\\
\leq & \, C_1 \, \rho^{2/\esp}
	\left(\frac{\mis({A_{h}})}{\mu(B_h)}\right)^{(\esp-1)/\esp}
	\!\!\!\!\! \frac{1}{\mu(B_h)} \frac{2^{3h+4}}{\rho^2}
	\int\!\!\int_{Q_{h-1}^-}(u-k_h)_+^2\,d\mu\,dt 
\end{align*}
where $C_1 = 2^{1/\esp} \, c_{\ast}^{2/\esp} (1 + \gamma + \gamma / (\theta - \bar\theta))$.
We also have that $u-k_h\leq m_+-k_h=\sigma_h\omega$ and then
$$
\int\!\!\int_{Q_{h-1}^-} (u-k_h)_+^2 \, d\mu dt \leq \mis (A_{h-1}) (\sigma_h \omega)^2 \leq
\mis (A_{h-1}) (\sigma \omega)^2 \,.
$$
This implies that
\begin{align*}
\int_{s_{h}}^s {{\int\!\!\!\!\!\!-}_{\!\!\!B_{h}}} (u - & k_h)_+^2 \zeta^2_h \,d\mu\,dt \leq 
	\, 2^{2h+4} \, C_1 (\sigma \omega)^2 \,  \frac{1}{\rho^{2\frac{\esp-1}{\esp}}}
	\left(\frac{\mis({A_{h}})}{\mu(B_h)}\right)^{\frac{\esp-1}{\esp}}
	\!\! \frac{\mis (A_{h-1})}{\mu(B_h)}								\\
& \leq \, 2^{2h+4} \, C_1 (\sigma \omega)^2 \,  \theta^\frac{\esp-1}{\esp} 2^{N(1 + \frac{\esp-1}{\esp})}
	\left(\frac{\mis({A_{h-1}})}{\mis(B_{h-1})}\right)^{\frac{\esp-1}{\esp}}
	\frac{\mis({A_{h-1}})}{\mu(B_{h-1})}
\end{align*}
where we have estimated $\mu(B_{h-1}) / \mu(B_h) \leq 2^N$.
By the last inequality and \eqref{primo}, if we call $C_2$ the constant 
$C_1 \,  \theta^\frac{\esp-1}{\esp} 2^{N(1 + \frac{\esp-1}{\esp})+6} (1-a)^{-2}$,
we obtain
$$
\frac{\mis ({A_{h+1}})}{\mu (B_{h+1})} \leq
	{C_2} \, 2^{4h} \, 
	\left(\frac{\mis({A_{h-1}})}{\mis(B_{h-1})}\right)^{\frac{\esp-1}{\esp}}
	\frac{\mis({A_{h-1}})}{\mu(B_{h-1})};
$$
finally, dividing by $s-s_{h+1}$ and since $(s-s_{h-1})/(s-s_{h+1}) \leq \theta/\bar\theta$,
we can summarize what we have obtain by writing
\begin{equation}
\label{ciebi}
y_{h+1} \leq C_3 \, 2^{4h} \, y_{h-1}^{1 + (\esp-1)/\esp}
\end{equation}
where we have defined
$$
y_h:=\frac{\mis ({A_{h}})}{\mis (Q_h^-)} \hskip20pt \textrm{and} \hskip20pt
C_3 = C_2 \frac{\theta}{\bar\theta}
$$
i.e.
$$
C_3 = 2^{1/\esp} \, c_{\ast}^{2/\esp} \left(1 + \gamma + \frac{\gamma}{\theta - \bar\theta} \right)
\, \theta^\frac{\esp-1}{\esp} 2^{N(1 + \frac{\esp-1}{\esp})+6} \frac{1}{(1-a)^{2}} \frac{\theta}{\bar\theta} \, .
$$
We observe that the hypotheses of Lemma \ref{lemmuzzo} are satisfied
with $c = C_3$, $b = 2^4$ and $\alpha = (\esp-1)/\esp$. Then if
\[
y_0 \leq c^{-1/\alpha} b^{-1/\alpha^2}
\]
we would be able to conclude, since $\{y_h\}_h$ is a decreasing sequence, that
\[
\lim_{h\to \infty} y_h = 0.
\]
Since $y_0 = {\mis ({A_{0}})}/{\mis (Q_0^-)}$, where
\[
Q^-_0 = B_{\rho}(y) \times (s - \theta \rho^2, s],\ \textrm{ and }\
A_0 = \{ (x,t) \in Q^-_0 : u(x,t) > m_+ - \sigma \om \}.
\]
To do this it is sufficient to choose $\nu_+$ to be
\[
\nu_+ =  C_3^{-\esp/(\esp-1)} 16^{-\esp^2/(\esp-1)^2} \, .
\]
By definition of $y_h$ and $A_h$  we see that
\[
u \leq m_+ - a \sigma \om \qquad
\mis\textrm{-a.e. in }B_{\rho/2}(y) \times \left(s - \bar\theta\rho^2,s\right],
\]
which completes the proof.
\end{proof}

An analogous argument proves the following claim.

\bprop
\label{degiorgi-}
Let $\rho,\,\theta > 0$ be
chosen such that the cylinder $Q_{\rho, \theta}^-(y,s)\subset\Omega
\times (0,T)$.  Then for each choice of $a, \sigma \in (0,1)$ and
$\bar\theta \in (0, {\theta})$, there is $\nu_-$, depending only on
$N, \gamma, c_{\ast}, a, \theta, \bar\theta$,
such that for every $u \in DG_-(\Omega, T, \gamma)$
and $m_+$ and $\om$ for which
$$
m_- \leq \essinf_{Q_{\rho, \theta}^-(y,s)} u
\quad\textrm{and}\quad
\om \geq \osc_{Q_{\rho, \theta}^-(y,s)} u,
$$
the following claim holds true: if
\[
\mis \left( \{ (x,t) \in Q_{\rho, \theta}^-(y,s): u(x,t) < m_- +\sigma \om \} \right)
	\leq \nu_-  \mis \left( Q_{\rho, \theta}^-(y,s) \right),
\]
then
$$
u(x,t) \geq m_- + a  \sigma  \om \hskip30pt
\mis\text{-a.e. in } B_{\rho/2}(y)\times (s-\bar\theta \rho^2,s].
$$
\eprop

\begin{proof}
  It is sufficient to argue as in the proof of
  Proposition~\ref{prop-DeGiorgi1} considering $(u-\hat k_h)_-$ in
  place of $(u-k_h)_+$, where $\hat k_h = m_- + \si_h \omega$.
\end{proof}

The next result is the so called \emph{expansion of
positivity}. Following the approach of DiBenedetto~\cite{Dib89Har}
we show that pointwise information in a ball $B_\rho$ implies
pointwise information in the expanded ball $B_{2\rho}$ at a further
time level.

\bprop
\label{esp_positivita}
Let $(x^{\ast}, t^{\ast}) \in \Om \times (0,T)$ and $\rho > 0$
with $B_{5\Lambda\rho}(x^{\ast})\times [t^\ast-\rho^2,t^\ast+\rho^2]\subset
\Om \times (0,T)$. Then there exists $\tilde\theta \in (0,1)$, depending only on $\gamma$, such
that for every $\hat\theta \in (0,\tilde\theta)$ there exists $\la \in
(0,1)$, depending on $\tilde\theta$ and $\hat\theta$, such that for every $h>0$ and  
for every $u \in DG(\Omega, T, \gamma)$ the following is valid. If 
\[
u(x,t^{\ast}) \geq h \qquad \mu-\textrm{a.e. in}\ B_{\rho}(x^{\ast}),
\]
then
\[
u(x,t) \geq \la h \qquad \mu-\textrm{a.e. in}\
B_{2\rho}(x^{\ast}), \qquad \mbox{for every } t \in [t^{\ast} + \hat\theta \rho^2, t^{\ast} + \tilde\theta \rho^2].
\]
\eprop

From now on, let us denote
\[ 
A_{h,\rho} (x^{\ast}, t^{\ast}) := \{ x \in B_{\rho}(x^{\ast}) :
u(x,t^{\ast}) < h \}.
\]

\boss
\label{Aacca_ro}
Let $(x^{\ast}, t^{\ast}) \in \Om \times (0,T)$ and $h > 0$ be fixed. Then if $u(x,t^{\ast}) \geq h$ for $\mu$-a.e.
$x \in B_{\rho}(x^{\ast})$ we have that
\[
A_{h,4\rho}(x^{\ast}, t^{\ast}) \subset B_{4\rho}(x^{\ast}) \setminus B_{\rho}(x^{\ast}).
\]
The doubling property implies
$$
\mu (A_{h,4\rho}(x^{\ast}, t^{\ast})) \leq \left(1-\frac1{4^N}\right)
\mu (B_{4\rho}(x^{\ast})) .
$$
\eoss

The proof of Proposition~\ref{esp_positivita} requires some preliminary lemmas.

\begin{lemma}
\label{lemma1}
Given $(x^{\ast} \!, t^{\ast})$ for which
$B_{4\rho}(x^{\ast}) \times [t^{\ast}, t^{\ast} + \theta \,  \rho^2] \subset \Omega \times (0,T)$,
there exist $\eta \in (0,1)$ and $\tilde\theta \in (0, \theta)$ such that,
given $h > 0$ and $u \geq 0$ in $DG(\Omega, T, \gamma)$ for which the following holds
$$
u(x,t^{\ast}) \geq h \hskip20pt \mu-\text{a.e. in } B_{\rho}(x^{\ast})
$$
then
$$
\mu (A_{\eta h, 4 \rho}(x^{\ast}, t)) <
\left( 1 - {\frac{1}{4^{N+1}}} \right) \mu (B_{4\rho}(x^\ast))
$$
for every $t \in [t^{\ast},t^{\ast} + \tilde\theta \rho^2]$.
\end{lemma}

\begin{proof}
  We may assume that $h=1$, otherwise we consider the scaled function
  $u/h$.  We apply the energy estimate of Definition \ref{classiDG}
  with $R=4\rho$, $r=4\rho(1-\sigma)$, $s_1 = t^{\ast}$,
$s_2 = t^{\ast} + \tilde\theta \rho^2$ with $\tilde\theta$ to be chosen, $\tau = t^{\ast}$,
$\si \in (0,1)$, and $\alpha = 1$.  This gives us
\begin{align*}
\sup_{t^{\ast} < t < t^{\ast} + \tilde\theta
    \rho^2} & \int_{B_{4\rho(1-\sigma)}(x^{\ast})} (u-1)_-^2 (x,t) d\mu
  (x) + \int_{t^{\ast} 
}^{t^{\ast} + \tilde\theta
    \rho^2}
  \int_{B_{4\rho(1-\sigma)}(x^{\ast})} g_{(u-1)_-}^2 \,d\mu\,dt 			 \\
& \leq \int_{B_{4\rho}(x^{\ast})} (u-1)_-^2
  (x,t^\ast) d\mu (x) + \frac{\gamma}{16 \sigma^2 \rho^2}
  \int_{t^{\ast}}^{t^{\ast} + \tilde\theta
    \rho^2}\!\!\!\!\int_{B_{4\rho}(x^{\ast})} (u-1)_-^2 \,d\mu\,dt.
\end{align*}
Since $u \geq 1 \hskip5pt \textrm{in }B_{\rho}(x^{\ast})$, we deduce
from Remark \ref{Aacca_ro} that
\[
\mu(\{ x \in B_{4\rho}(x^{\ast}): u(x,t^{\ast}) < 1 \}) < \left( 1 - \frac1{4^N}
\right) \mu(B_{4\rho}(x^{\ast})) .
\]
Notice that $(u-1)_- \leq 1$; thus we have in particular
\begin{align*}
\sup_{t^{\ast} < t < t^{\ast} + \tilde\theta
    \rho^2} &
  \int_{B_{4\rho(1-\sigma)}(x^{\ast})} (u-1)_-^2 (x,t) \,d\mu(x) \\
& \leq \int_{B_{4\rho}(x^{\ast})} (u-1)_-^2
  (x,t^\ast) \,d\mu (x) + \frac{\gamma}{16 \sigma^2 \rho^2}
  \int_{t^{\ast}}^{t^{\ast} + \tilde\theta \rho^2}\!\!\!\!\int_{B_{4\rho}(x^{\ast})} (u-1)_-^2 \,d\mu\,dt \\
& \leq \left( 1 - \frac{1}{4^{N}} \right)
  \mu(B_{4\rho}(x^{\ast})) + \frac{\gamma}{16 \sigma^2 \rho^2}
 \int_{t^{\ast}}^{t^{\ast} + \tilde\theta \rho^2}\!\!\!\!
  \int_{B_{4\rho}(x^{\ast})} (u-1)_-^2 \,d\mu\,dt \\
& \leq \left( 1 - \frac{1}{4^{N}} \right)
  \mu(B_{4\rho}(x^{\ast})) + \frac{\gamma \tilde\theta}{16 \sigma^2}
  \mu(B_{4\rho}(x^{\ast}))
\end{align*}
Writing $A_{h,\rho}(t)$ in place of $A_{h,\rho}(x^{\ast}, t)$,
decomposing
\[
A_{\eta,4\rho}(t) = A_{\eta,4\rho(1-\si)} (t) \cup \{x \in
B_{4\rho}(x^{\ast}) \setminus B_{4\rho(1-\sigma)}(x^{\ast}): u(x,t) <
\eta \},
\]
and using the doubling property we have
\begin{align*}
  \mu(A_{\eta,4\rho}(t)) & \leq \mu(A_{\eta,4\rho(1-\sigma)} (t)) +
  \mu \big(B_{4\rho}(x^{\ast}) \setminus
  B_{4\rho(1-\sigma)}(x^{\ast})\big).
\end{align*}
On the other hand,
\[
\int_{B_{4\rho(1-\sigma)}(x^{\ast})} (u-1)_-^2 (x,t)\, d\mu(x)
\geq \int_{A_{\eta, 4\rho(1-\si)}(t)} (u-1)_-^2(x,t)\, d\mu(x)
\geq (1-\eta)^2 \mu(A_{\eta, 4\rho(1-\si)}(t)).
\]
Finally, we obtain
\begin{align} \label{stimasigma} \mu(A_{\eta,4\rho} & (t)) \leq
  \mu(A_{\eta,4\rho(1-\si)}(t)) +
  \mu \left(B_{4\rho}(x^{\ast}) \setminus B_{4\rho(1-\sigma)}(x^{\ast})\right) \\
  \leq & (1-\eta)^{-2} \int_{B_{4\rho(1-\sigma)}(x^{\ast})} (u-1)_-^2
  (x,t) d\mu
  + \mu \left(B_{4\rho}(x^{\ast}) \setminus B_{4\rho(1-\sigma)}(x^{\ast})\right) \nonumber \\
  & \leq (1-\eta)^{-2} \left( 1 - \frac{1}{4^{N}} + {\frac{\gamma
        \tilde\theta}{16 \si^2}} \right) \mu(B_{4\rho}(x^{\ast})) + \mu
  \left(B_{4\rho}(x^{\ast}) \setminus
    B_{4\rho(1-\sigma)}(x^{\ast})\right). \nonumber
\end{align}
If the claim of the lemma was false, then for every $\tilde\theta, \eta \in
(0,1)$ there exists $\bar{t} \in [t^\ast, t^\ast + \tilde\theta\rho^2]$ for which
\[
\mu(A_{\eta,4\rho}(\overline{t})) \geq
\left( 1 - \frac1{4^{N+1}} \right) \mu \left( B_{4\rho}(x^{\ast}) \right).
\]
Applying this last estimate, then \eqref{stimasigma} for $t = \bar{t}$ and dividing by
$\mu \left( B_{4\rho}(x^{\ast}) \right)$ we would have
\begin{align*}
\left( 1 - \frac1{4^{N+1}}\right)
  \leq (1-\eta)^{-2} \left(1 - \frac1{4^N} + \frac{\gamma \tilde \theta}{16\si^2}\right)
    + \frac{ \mu(B_{4\rho}(x^{\ast}) \setminus B_{4\rho(1-\sigma)}(x^{\ast}))}{\mu( B_{4\rho}(x^{\ast}))} .
\end{align*}
Choosing, for instance,
$\tilde\theta = \si^3$ and letting $\eta$ and $\si$ go to zero we would have the contradiction
$1 - 4^{-N-1} \leq 1 - 4^{-N}$.
\end{proof}

\begin{lemma}
\label{lemma2}
Assume $u\in DG(\Omega,T,\gamma)$, $u\geq 0$. Let $\tilde\theta$ be as in
Lemma $\ref{lemma1}$ and $h > 0$. Consider $(x^{\ast}, t^{\ast})$ in such a way that 
$B_{5\Lambda\rho}(x^\ast)\times [t^\ast-\tilde \theta\rho^2,t^\ast+\tilde \theta\rho^2]\subset \Omega \times (0,T)$ and
assume that
\[
u(x,t^\ast)\geq h, \qquad \mu-\mbox{a.e. } x\in B_\rho(x^\ast).
\]
Then for every $\e > 0$ there exists $\eta_1\in (0,1)$, depending on $\e$, $c_d$, $\gamma$, $\tilde\theta$, 
and the constant in the weak Poincar\'e inequality, 
such that
\[
\mis \left(\{ (x,t)\in B_{4\rho}(x^\ast)\times [t^\ast,t^\ast + \tilde\theta\rho^2] : u(x,t) <
  \eta_1 h\}\right) < \e \mis
\left(B_{4\rho}(x^\ast)\times [t^\ast,t^\ast + \tilde\theta\rho^2]\right) .
\]
\end{lemma}

\begin{proof}
  Apply the energy estimate \eqref{DGgamma} in
  $B_{5\Lambda\rho}(x^\ast)\times (t^\ast - 2 \tilde\theta\rho^2,t^\ast)$ with
\[
R = 5\Lambda\rho,\ r = 4\Lambda\rho,\ s_2 = t^{\ast}+ \tilde\theta \rho^2,
\ s_1 = t^{\ast}-\tilde \theta\rho^2,\ \tau = t^{\ast},\ \textrm{ and }\ \alpha = 0,
\]
at the level $k= \eta h 2^{-m}$, where $\eta > 0$ and $m \in \N$. We obtain
\begin{align} \label{forniscono}
  \int_{t^\ast}^{t^\ast+\tilde\theta \rho^2}& \!\!\int_{B_{4\Lambda\rho}(x^\ast)}
		g_{(u - \frac{\eta h}{2^m})_-}^2 \,d\mu\,dt \\
  & \leq \gamma \left( 1 + \frac1{2\tilde\theta} \right)
  \frac{1}{(\Lambda\rho)^2} \int_{t^\ast-\tilde \theta\rho^2}^{t^\ast+ \tilde \theta \rho^2}
  \!\!\int_{B_{5\Lambda\rho}(x^{\ast})}
  \left(u-\frac{\eta h}{ 2^m}\right)_-^2 \,d\mu\,dt \nonumber \\
  & \leq \gamma \left( 1 + \frac1{2\tilde\theta} \right)
  \frac{1}{(\Lambda\rho)^2}
  {\frac{\eta^2h^2}{2^{2m}}}  2 \tilde\theta \rho^2 \mu\left(B_{5\Lambda\rho}(x^\ast)\right) \nonumber \\
  & \leq \gamma( 2 \tilde\theta + 1)\frac{\eta^2h^2}{2^{2m}}
  \mu\left(B_{5\Lambda\rho}(x^\ast)\right). \nonumber
\end{align}
To simplify notation, let us write $A_{h,\rho}(t)$ instead of
$A_{h,\rho}(x^\ast, t)$. Lemma \ref{lemma2.2} with parameters $k = \eta
h/2^m$, $l = \eta h/2^{m-1}$, $q=1$ and $2-\varepsilon <p< 2$, implies
\begin{align} \label{servesubito} \int_{B_{4\rho}(x^\ast)}
  \left(u-\frac{\eta h }{2^{m}}\right)_-(x,t)\, d\mu & \leq
  \frac{\eta h }{2^{m}} \mu(A_{\eta h 2^{-m},4\rho}(t))  \\
  & \leq \frac{8C_p\rho\mu(
    B_{4\rho}(x^\ast))^{2-1/p}}{\mu(B_{4\rho}(x^\ast) \setminus A_{\eta h
      2^{-m+1},4\rho}(\tau))} \biggl(\int_{\tilde{A}(t)}g_{(u -
    \frac{\eta h}{2^m})_-}^{p}(x,t) \,d\mu\biggr)^{1/p}, \nonumber
\end{align}
for every $t \in [t^\ast,t^\ast+\tilde \theta\rho^2]$, where $\tilde{A}(t):=
A_{\eta h 2^{-m+1},4\Lambda\rho}(t) \setminus A_{\eta h
  2^{-m},4\Lambda\rho}(t)$. Clearly,
\[
B_{4\rho}(x^\ast)\setminus A_{\eta h 2^{-m+1},4\rho}(t)\supseteq B_{4\rho}(x^\ast)\setminus A_{\eta h,4\rho}(t).
\]
If we choose $\eta$ so that it satisfies the hypothesis of
Lemma~\ref{lemma1} and write
\[
\mu(B_{4\rho}(x^\ast) \setminus A_{\eta h,4\rho}(t)) + \mu(A_{\eta
  h,4\rho}(t)) = \mu(B_{4\rho}(x^\ast)),
\]
then we deduce that
\[
\mu(B_{4\rho}(x^\ast)\setminus A_{\eta h,4\rho}(t)) > 4^{-N-1}
\mu(B_{4\rho}(x^\ast))
\]
for every $t \in [t^\ast,t^\ast+\tilde \theta\rho^2]$. We finally arrive at
\[
\int_{B_{4\rho}(x^{\ast})} \left(u-\frac{\eta h}{2^m}\right)_-(x,t)\, d\mu
\leq 8 C_{p}4^{N+1}\mu(B_{4\rho}(x^\ast))^{1-1/p}\rho\biggl(
\int_{\tilde{A}(t)} g_{(u -
    \frac{\eta h}{2^m})_-}^{p}(x,t)\, d\mu\biggr)^{1/p}.
\]
Integrating this with respect to $t$ and defining the decreasing
sequence $\{a_{m,\rho}\}_{m=0}^\infty$ as
\begin{align*}
a_{m,\rho} : & = \dint_{t^{\ast}}^{t^{\ast}+\tilde\theta\rho^2} \mu \left( A_{\eta h 2^{-m},\rho}(t) \right) \,d t \\
    & = \mis \biggl( \biggl\{ (x,t) \in B_\rho(x^{\ast}) \times [t^{\ast} - \tilde\theta\rho^2,t^{\ast}]: u(x,t) < 
\frac{\eta h}{2^{m}}\biggr\} \biggr),
\end{align*}
we get by the H\"older inequality
\begin{align}\label{intau}
  & \int_{t^{\ast} }^{t^{\ast}+\tilde\theta \rho^2} \!\!\int_{B_{4\rho}(x^{\ast})} 
	\biggl(u-\frac{\eta h}{2^m}\biggr)_- (x,t)\, d\mu\, d t \\
  & \leq 8C_{p}4^{N+1}\mu(B_{4\rho}(x^\ast))^{1-1/p}\rho\biggl(\int_{t^\ast }^{t^\ast+\tilde\theta
    \rho^2} \!\! \int_{\tilde{A}(t)} g_{(u -
    \frac{\eta h}{2^m})_-}^{p}(x,t)\,d\mu\,d t\biggr)^{1/p}  \nonumber \\
  & \leq 8C_{p}4^{N+1}\mu(B_{4\rho}(x^\ast))^{1-1/p}\rho
  \biggl(\int_{t^\ast}^{t^\ast+\tilde\theta \rho^2}
  \!\!\int_{B_{4\Lambda\rho}(x^\ast)} g_{(u-\frac{\eta
      h}{2^m})_-}^2\,d\mu\, d t\biggr)^{1/2} (a_{m-1,4\Lambda\rho} -
  a_{m,4\Lambda\rho})^{(2-p)/2}. \nonumber
\end{align}
On the other hand,
\[
\int_{t^\ast }^{t^\ast+\tilde\theta \rho^2} \!\!\int_{B_{4\rho}(x^\ast)} \biggl(u
-\frac{\eta h}{ 2^{m}}\biggr)_- (x,t) \, d\mu \,d t \geq
\frac{\eta h}{ 2^{m+1}}a_{m+1,4\rho}
\]
from which, using first \eqref{intau} and then \eqref{forniscono}, we obtain
\[
a_{m+1,4\rho}^{2/(2-p)} \leq c(a_{m-1,4\Lambda\rho} - a_{m,4\Lambda\rho}),
\]
where $c = (C_{p}^2 16^{N+3}\gamma (2 \tilde\theta+1)\mu(B_{4\rho}(x^\ast))^{2(1-1/p)}
\mu(B_{5\Lambda\rho}(x^\ast))\rho^2)^{1/(2-p)}$. Hence for every $m_\ast \in {\bf N}$ we have
\[
\sum_{m=1}^{m_\ast}a_{m+1,4\rho}^{2/(2-p)} \leq c(a_{0,4\Lambda\rho}
- a_{m_\ast,4\Lambda\rho}).
\]
Since $\{a_{m,\rho}\}_{m=0}^\infty$ is decreasing the sum
$\sum_{m=1}^{\infty}a_{m+1,4\rho}^{2/(2-p)}$ converges, and consequently
$$\lim_{m\to\infty} a_{m,4\rho} = 0.$$ This completes the proof.
\end{proof}

\begin{proof}[Proof of Proposition~\ref{esp_positivita}]
The proof is a direct consequence of Proposition~\ref{degiorgi-} used with the right parameters.  
Fix $\theta=1$ and let $\tilde\theta$ be as in Lemma~\ref{lemma1}; choose also $\hat\theta\in (0,\tilde\theta)$
and let $\nu_-$ be the constant in Proposition~\ref{degiorgi-} determined by these parameters and $a=1/2$.  
Apply Lemma~\ref{lemma2} with $\epsilon=\nu_-$ and obtain the constant $\eta_1$ for which the 
assumptions of Proposition~\ref{degiorgi-} are satisfied with 
\[
y = x^{\ast},\,s = t^{\ast} + \tilde\theta\rho^2, \bar\theta := \tilde\theta -\hat\theta, 
m_-=0 \mbox{ and } \sigma=\frac{\eta_1h}{\omega}.
\]
This concludes the proof with $\lambda=\frac{1}{2}\eta_1$.
\end{proof}

The following is the main result of this paper.

\begin{theorem}[Parabolic Harnack]\label{Harnack}
Assume $u\in DG(\Omega,T,\gamma)$, $u\geq 0$. For any constant
$c_2\in (0,1]$, there exists $c_1>0$, depending on $c_d$, $\gamma$, and the constants
in the weak $(1,2)$--Poincar\'e inequality, such that for
any Lebesgue point $(x_0,t_0)\in \Omega\times (0,T)$ with 
$B_{5\Lambda\rho}(x_0)\times (t_0-\rho^2,t_0+5\rho^2)\subset
\Omega\times (0,T)$ we have
$$
u(x_0,t_0)\leq c_1 \essinf_{B_\rho(x_0)} u(x,t_0+c_2\rho^2).
$$
As a consequence, $u$ is locally $\alpha$-H\"older continuous with
$\alpha=-\log_2 \frac{1-\gamma}{\gamma}$ and satisfies the strong
maximum principle.
\end{theorem}

\begin{proof}
Suppose $t_0=0$; up to rescaling, we may write $u(x_0,0)=\rho^{-\xi}$ for
some $\xi>0$ to be fixed later. Define the
functions
\[
{\mathcal M}(s)=\sup_{Q_s^-(x_0,0)} u,\qquad
{\mathcal N}(s)=(\rho-s)^{-\xi}, \qquad s\in [0,\rho).
\]
Let us denote by $s_0\in [0,\rho)$ the largest solution
of ${\mathcal M}(s)={\mathcal N}(s)$.
Define
\[
M:={\mathcal N}(s_0)= (\rho-s_0)^{-\xi},
\]
and fix $(y_0,\tau_0)\in Q^-_{s_0}(x_0,0)$ in such a way that
\begin{equation}\label{choicey0t0}
\frac{3M}{4} < \sup_{Q^-_{\rho_0/4}(y_0,\tau_0)}u
\leq M,
\end{equation}
where $\rho_0=(\rho-s_0)/2$; this implies that
$Q^-_{\rho_0}(y_0,\tau_0)\subset Q^-_{(\rho+s_0)/2}(x_0,0)$,
as well as that
$$
\sup_{Q^-_{\rho_0}(y_0,\tau_0)} u\leq
\sup_{Q^-_{(\rho+s_0)/2}(x_0,0)} u
<{\mathcal N}\left(\frac{\rho+s_0}{2}\right)=2^\xi M.
$$

Let us divide the proof into five steps.

{\it Step 1.}
We assert that
\begin{equation}\label{eq2.19}
\mis \left( \left\{
(x,t)\in Q^-_{\rho_0/2}(y_0,\tau_0): u(x,t)> \frac{M}{2}
\right\}\right) > \nu_+ \mis \left(
Q^-_{\rho_0/2}(y_0,\tau_0)\right),
\end{equation}
where $\nu_+$ is the constant in Proposition~\ref{prop-DeGiorgi1}. To see this,
assume on the contrary that equation~\eqref{eq2.19} is not true. Then set
$k=2^\xi M$ and 
\[
m_+=\omega=k, \,\theta=1,
\,\rho=\frac{\rho_0}{2},\,\sigma=1-2^{-\xi-1},\, \textrm{ and }\,
a=\sigma^{-1}\biggl(1-\frac{3}{2^{\xi+2}}\biggr).
\]
We obtain from Proposition~\ref{prop-DeGiorgi1} that
\[
u\leq \frac{3M}{4} \quad \textrm{in }\, Q^-_{\frac{\rho_0}{4}}(y_0,\tau_0),
\]
which contradicts \eqref{choicey0t0}.

{\it Step 2.}
We show that there exists
$$
\bar t\in \left(\tau_0-\frac{\rho_0^2}4,\tau_0
-\frac{\nu_+}{8}\frac{\rho_0^2}{4}\right]
$$
such that
\begin{equation}\label{estnuM}
\mu \left(
\left\{
x\in B_{\rho_0/2}(y_0): u(x,\bar t)\geq \frac{M}{2}
\right\}
\right)> \frac{\nu_+}{2}\mu(B_{\rho_0/2}(y_0)),
\end{equation}
and
\begin{equation}\label{estnuM2}
\int_{B_{\rho_0/2}(y_0)} g_u^2(x,\bar t)\,d\mu(x)\leq \alpha
\frac{\mu(B_{\rho_0}(y_0))}{\rho_0^2}
k^2,
\end{equation}
for some sufficiently large $\alpha>0$. For this, we define the
sets $A(t)$, $I$, and $J_\alpha$ as follows
$$
A(t):=\left\{
x\in B_{\rho_0/2}(y_0): u(x,t)\geq \frac{M}{2}
\right\},
$$
$$
I:=\left\{
t\in (\tau_0-\frac{\rho_0^2}{4},\tau_0] : \mu(A(t))>
\frac{\nu_+}{2}\mu(B_{\rho_0/2}(y_0))
\right\},
$$
and
$$
J_\alpha :=\left\{
t\in (\tau_0-\frac{\rho_0^2}{4},\tau_0] : \int_{B_{\rho_0/2}(y_0)}
g_u^2(x,t)d\mu(x)\leq \alpha
\frac{\mu(B_{\rho_0}(y_0))}{\rho_0^2}
k^2
\right\}.
$$
From \eqref{eq2.19} we have that
\begin{align*}
\nu_+\mis (Q^-_{\rho_0/2}(y_0,\tau_0))<&
\int_{\tau_0-\rho_0^2/4}^{\tau_0} \mu(A(t))\,dt \\
& =
\int_I \mu(A(t))dt +\int_{(\tau_0-\rho_0^2/4]\setminus I} \mu(A(t))\,dt \\
& \leq
\mu(B_{\rho_0/2}(y_0))|I|+\frac{\nu_+}{2}
\mis (Q^-_{\rho_0/2}(y_0,\tau_0))\\
& =
\mis (Q^-_{\rho_0/2}(y_0,\tau_0)) \left(
|I| \left( \frac{4}{\rho_0^2}\right) +\frac{\nu_+}{2}
\right).
\end{align*}
This implies the following lower bound
$$
|I|\geq \frac{\nu_+\rho_0^2}{8}.
$$
On the other hand, if we apply \eqref{DGgamma} 
with $R=\rho_0$, $r=\rho_0/2$, $\alpha=0$, and $\theta=1$, we obtain
\begin{align}\label{est2.18}
\int_{Q^-_{\rho_0/2}(y_0,\tau_0)}g_u^2 \,d\mu\, dt &  =
\int_{Q^-_{\rho_0/2}(y_0,\tau_0)}g_{(u-k)_-}^2 \,d\mu\, dt \\
& \leq
\frac{8\gamma}{\rho_0^2}
\int_{Q^-_{\rho_0}(y_0,\tau_0)} (u-k)_-^2 \,d\mu\, dt
\leq
\frac{8\gamma k^2}{\rho_0^2} \mis (Q^-_{\rho_0}(y_0,\tau_0)) \nonumber \\
&=8\gamma k^2 \mu (B_{\rho_0}(y_0)). \nonumber
\end{align}
This estimate implies
\begin{align*}
4\gamma k^2\mu(B_{\rho_0}(y_0))
&\geq
\int_{(\tau_0-\rho_0^2/4,\tau_0]} \,dt
\int_{B_{\rho_0/2}(y_0)} g_u^2(x,t)\,d\mu \\
& \geq
\alpha \frac{\mu(B_{\rho_0}(y_0))}{\rho_0^2}k^2
\left( \frac{\rho_0^2}{4}-|J_\alpha|\right),
\end{align*}
which in turn gives us
$$
|J_\alpha|\geq \frac{\rho_0^2}{4}\left(1-\frac{16\gamma}{\alpha}\right).
$$
Choosing $\alpha=64\gamma/{\nu_+}$, we obtain
$$
|I\cap J_\alpha|=|I|+|J_\alpha|-|I\cup J_\alpha|
\geq \frac{\nu_+\rho_0^2}{16}.
$$
Then if we set
$$
T=\left(\tau_0-\frac{\rho_0^2}{4},\tau_0-\frac{\nu_+}{8}\frac{\rho_0^2}{4}\right],
$$
we get
$$
|I\cap J_\alpha\cap T|=|I\cap J_\alpha|+|T|-|(I\cap J_\alpha)\cup T|
\geq \frac{\rho_0^2}{4}\frac{\nu_+}{8},
$$
and in particular $I\cap J_\alpha\cap T\neq \emptyset$.

{\it Step 3.}
We fix $\bar t\in T$; by Lemma~\ref{lemmaMisVar}
we have that for any $\delta \in (0,1)$,
there exist $x^\ast\in B_{\rho_0/2}(y_0)$ and $\eta\in (0,1)$ such that
\begin{equation}\label{estMis}
\mu\left(\left\{u(\cdot,\bar t)>\frac{M}{4}\right\}
\cap B_{\eta \rho_0/2}(x^\ast)\right) >
(1-\delta) \mu(B_{\eta\rho_0/2}(x^\ast)).
\end{equation}

{\it Step 4.}
We show that for $\varepsilon>0$ to be fixed, there exists $\bar x$  such
that $Q^+_{\varepsilon\eta\rho_0/4}(\bar x,\bar t)\subset Q^-_{\rho_0}(y_0,\tau_0)$
and
\begin{equation}\label{condSmall}
\mis \left(
\left\{u\leq \frac{M}{8}\right\} \cap Q^+_{\varepsilon\eta\rho_0/4}(\bar x,\bar t)
\right)\leq
4^{N+1}(\gamma \varepsilon^2+\delta)
\mis \left(Q^+_{\varepsilon\eta\rho_0/4}(\bar x,\bar t)\right).
\end{equation}
Indeed, consider the cylinder
$$
Q=B_{\eta\rho_0/4}(x^\ast)\times (\bar t,\bar t+t^\ast]
$$
with $t^\ast=(\varepsilon\eta \rho_0/4)^2$. Using the energy estimate \eqref{DGgamma} on $Q$ with
$k=M/4$, $R=\eta\rho_0/2$, $r=R/2$ and $\alpha=1$, we obtain together
with \eqref{estMis} that for any $s\in (\bar t,\bar t+t^\ast]$
\begin{align*}
  & \int_{B_{\eta\rho_0/4}(x^\ast)}\biggl(u-\frac{M}{4}\biggr)^2_- (x,s)\, d\mu(x) \\
  & \leq \frac{16\gamma}{\eta^2\rho_0^2} \int_{\bar t}^{{\bar t}+t^\ast}
  dt \int_{B_{\eta\rho_0/2}(x^\ast)}\left(u-\frac{M}{4}\right)^2_-(x,t)
  \,d\mu(x) +
  \int_{B_{\eta\rho_0/2}(x^\ast)}\left(u-\frac{M}{4}\right)^2_-
  (x,\bar t)\, d\mu(x)\\
  & \leq \frac{M^2}{16}(\gamma\varepsilon^2 +\delta)
  \mu(B_{\eta\rho_0/2}(x^\ast)).
\end{align*}
Define
$$
B(t)=\left\{ x\in B_{\eta\rho_0/4}(x^\ast): u(x,t)\leq \frac{M}{8}\right\},
$$
and we have that
$$
\int_{B_{\eta\rho_0/4}(x^\ast)} \left(u-\frac{M}{4}\right)^2_-(x,s)d\mu(x)
\geq \int_{B(s)}\left(u-\frac{M}{4}\right)^2_-(x,s)d\mu(x)
\geq \frac{M^2}{64}\mu(B(s)).
$$
Putting the preceding two estimates together we arrive at
$$
\mu(B(s))\leq 4 (\gamma\varepsilon^2 +\delta) \mu(B_{\eta\rho_0/2}(x^\ast))
$$
for every $s\in (\bar t,\bar t+t^\ast]$. Integrating this inequality
over $s$ we obtain the estimate
$$
\mis \left(
\left\{ u\leq \frac{M}{8}\right\} \cap
Q^+_{\varepsilon\eta\rho_0/4}(x^\ast,\bar t)
\right)
\leq
2^{N+2} (\gamma \varepsilon^2+\delta)
\mis \left(Q^+_{\varepsilon\eta\rho_0/4}(x^\ast,\bar t)\right).
$$
We have to apply Proposition~\ref{prop-DeGiorgi1}; we then have that there exists 
$\bar x$ so that $Q^+_{\varepsilon\eta\rho_0/4}(\bar x,\bar t)\subset Q^-_{\rho_0}(y_0,\tau_0)$
satisying equation~\eqref{condSmall}. To see this, we take a disjoint family of balls
$\{B_{\varepsilon\eta\rho_0/4}(x_j)\}_{j=1}^m$
such that $B_{\varepsilon\eta\rho_0/4}(x_j)
\subset B_{\eta\rho_0/4}(x^\ast)$ for every $j=1,\ldots,m$,
and
$$
B_{\eta\rho_0/4}(x^\ast)\subset
\bigcup_{j=1}^m B_{\varepsilon\eta\rho_0/2}(x_j).
$$
Given this disjoint family, there exists $j_0$ such that
\eqref{condSmall} is satisfied with $\bar x=x_{j_0}$. Otherwise we
would get a contradiction summing over $j=1,\ldots,m$.

{\it Step 5.}
Due to our construction, we are able to state
\[
\osc_{Q^+_{\varepsilon\eta\rho_0/4}(\bar x,\bar t)} u\leq k
=2^\xi M.
\]
We also have that if $\bar s=\bar t+(\varepsilon \eta\rho_0/4)^2$ then
$Q^+_{\varepsilon\eta\rho_0/4}(\bar x,\bar t)
=Q^-_{\varepsilon\eta\rho_0/4}(\bar x,\bar s)$; we apply
Proposition~\ref{degiorgi-} with
$$
\rho=\frac{\varepsilon\eta\rho_0}{4},\quad \theta=1,\quad m_-=0, \quad
\omega=k,\quad a=\frac{1}{2}, \quad \sigma=2^{-\xi-3},
$$
so we can deduce that there exists $\nu_->0$ such that if
\begin{equation}\label{condPropDeGiorgi1}
\mis \biggl(
\biggl\{u\leq \frac{M}{8}\biggr\}\cap Q^-_{\varepsilon\eta\rho_0/4}(\bar x,\bar s)
\biggr)\leq \nu_-\mis (Q^-_{\varepsilon\eta\rho_0/4}(\bar x,\bar s))
\end{equation}
then
$$
u(x,t)\geq \frac{M}{16}, \qquad \mis-\textrm{a.e. in}\ Q^-_r(\bar x,\bar s),
$$
where $r=\varepsilon\eta\rho_0/8$.

Fix $\varepsilon$ and $\delta$ in \eqref{condSmall} small enough so
that \eqref{condPropDeGiorgi1} is satisfied and $\bar t+(\varepsilon
\eta \rho/4)^2<0$.  With this choice of $\delta$, we obtain the
constants $\eta$ and $r$ that depend only on
$\delta$. Expansion of positivity, Proposition~\ref{esp_positivita}, implies
$$
u(x,t)\geq \lambda \frac{M}{16},
$$
for all $x\in B_{2r}(\bar x)$ and $t\in [\hat t+\hat \theta r^2,\hat
t+\tilde\theta r^2]$ for some $\hat t\in (\bar t,\bar
t+(\varepsilon\eta\rho_0/4)^2]$, where $\tilde\theta$ depends only on
$\gamma$, whereas $\lambda$ depends on $\gamma$ and $\hat \theta\in
(0,\tilde\theta)$ that we shall fix later. We can repeat the argument 
with $r$ replaced by $2r$ and initial time varying in the interval
$[\hat t+\hat \theta r^2,\hat t+\tilde\theta r^2]$ to
obtain the following estimate
$$
u(x,t)\geq \lambda^2 \frac{M}{16},
$$
for all $x\in B_{4r}(\bar x)$ and $t\in [\hat t+5\hat \theta
r^2,\hat t+5\tilde\theta r^2]$.  Thus iterating this procedure, we can show by
induction that for any $m\in \N$
\begin{equation}\label{estimateM}
u(x,t)\geq \lambda^m \frac{M}{16},
\end{equation}
for all $x\in B_{2^mr}(\bar x)$ and $t\in [s_m,t_m]$, where
$$
s_m=\hat t+\hat \theta r^2 \frac{4^m-1}{3}
\quad\text{and}\quad
t_m=\hat t+\tilde\theta r^2\frac{4^m-1}{3}.
$$
We fix $m$ in such a way that $2\rho < 2^m r\leq 4\rho$;
since $\bar x\in B_{\rho}(x_0)$, we then have the inclusion
$B_{\rho}(x_0)\subset B_{2^mr}(\bar x)$. Recalling that 
$r=\varepsilon\eta(\rho-s_0)/16$, we obtain
\begin{align*}
  (\rho-s_0)^{-\xi}=& \left(\frac{2^4}{\varepsilon \eta} r
  \right)^{-\xi} =\frac{(\varepsilon \eta)^\xi}{2^{4\xi}r^\xi} \geq
  (\varepsilon \eta)^\xi 2^{(m-6)\xi}\rho^{-\xi}.
\end{align*}
Hence equation \eqref{estimateM} can be rewritten as follows
$$
u(x,t)\geq \lambda^m \frac{M}{16}=\lambda^m \frac{(\rho-s_0)^{-\xi}}{16}
\geq (2^\xi\lambda)^m (\varepsilon\eta)^\xi 2^{-6\xi-4}\rho^{-\xi}
=(2^\xi\lambda)^m (\varepsilon\eta)^\xi 2^{-6\xi-4}u(x_0,0).
$$
for any $x\in B_{\rho}(x_0)$ and $t\in [s_m,t_m]$.

We now fix $c_2>0$ and choose
$\hat \theta$ in such a way that $\frac{16}{3}\hat
\theta<c_2$.  With this choice, since $2^mr\leq 4\rho$, we have
\begin{equation}\label{eqsmc2}
s_m\leq\hat \theta r^2 \frac{4^m}{3}\leq \hat \theta \frac{16}{3}\rho^2<c_2\rho^2.
\end{equation}
Once $\hat \theta$ has been fixed, we have $\lambda$;
we now fix $\xi=-\log_2 \lambda$. With these choices
also the radius $r$ is fixed and so $m$ is chosen in such a way that
$$
1-\log_2 r\leq m\leq 2-\log_2 r.
$$
We draw the conclusion that
$$
u(x,t)\geq c_0 u(x_0,0)
$$
with $c_0:=(\varepsilon \eta)^\xi 2^{-6\xi-4}$ for all $x\in B_{\rho}(x_0)$ and $t\in [s_m,t_m]$.

Notice that by \eqref{eqsmc2} we have got two alternatives. Either
$c_2\rho^2\in [s_m,t_m]$ or $c_2\rho^2>t_m$. In the former case, the
proof is completed by taking $c_1=c_0^{-1}$. Whereas in the latter
case, we can select $\tilde t\in [s_m,t_m]$ such that
$$
u(x,\tilde t)\geq c_0 u(x_0,0)
$$
for all $x\in B_{\rho}(x_0)$.
We can assume that $\hat \theta$ is small enough such that $\tilde t+\hat \theta\rho^2
<c_2\rho^2$.
By expansion of positivity, Proposition \ref{esp_positivita}, we then obtain that
$$
u(x,t)\geq \lambda c_0 u(x_0,0)
$$
for all $x\in B_{2\rho}(x_0)$ and $t\in [\tilde
t+\hat\theta\rho^2,\tilde t+\tilde\theta\rho^2]$.  If $c_2\rho^2<\tilde
t+\tilde\theta\rho^2$, then the proof is completed by selecting $c_1=(\lambda
c_0)^{-1}$. If this was not the case, we could restrict the previous
inequality on $B_{\rho}(x_0)$, and so iterating the procedure, adding
the condition that $\hat \theta\leq \tilde\theta$, using the fact that the
estimate is already true on $[\tilde t+\hat\theta\rho^2,\tilde
t+\tilde\theta\rho^2]$,
$$
u(x,t)\geq \lambda^2 c_0 u(x_0,0)
$$
for each $x\in B_{\rho}(x_0)$ and $t\in [\tilde
t+\hat\theta\rho^2,\tilde t+2\tilde\theta\rho^2]$.  By induction, if $k$
is an integer such that $\tilde t+k\tilde\theta\rho^2\geq c_2\rho^2$, then
$$
u(x,t)\geq \lambda^k c_0 u(x_0,0)
$$
for every $x\in B_{\rho}(x_0)$ and $t\in [\tilde
t+\hat\theta\rho^2,\tilde t+k\tilde\theta\rho^2]$. It is crucial to select
such an index $k$ which depends only on the class and not on the
function. We then take $k$ in such a way that $\tilde t+k\tilde\theta\geq
c_2\rho^2$. As $-\rho^2\leq \tilde t\leq c_2\rho^2$ the index $k$ has
to be chosen in such a way that both $1+c_2\leq k\tilde\theta$ and $\tilde
t+k\tilde\theta$ remains in the domain of reference. Notice that $1+c_2\leq
2$. Hence there exists $k$ such that $2\leq k\tilde\theta\leq 3$, and we are
done with the proof.
\end{proof}

\end{document}